\documentclass{article}
\usepackage{amssymb}
\usepackage{color}
\newtheorem{theorem}{THEOREM}[section]
\newtheorem{proposition}{PROPOSITION}[section]
\newtheorem{lemma}{LEMMA}[section]

\newtheorem{remark}{REMARK}[section]

\newcommand{\grad}{\bigtriangledown}
\begin{document}
\title{Bifurcation of radial solutions for prescribed mean curvature equations}
\author{\Large N. B. Zographopoulos
}
\date{}
\maketitle
\pagestyle{myheadings} \thispagestyle{plain} \markboth{N. B.
Zographopoulos}{Global bifurcation for mean curvature equations}
\maketitle
\begin{abstract}
We prove global bifurcation results for prescribed mean curvature equations. These equations are defined on $\mathbb{R}^3$ and the radial solutions belonging in these branches are smooth and positive.
\end{abstract}
%
%
%
%
%
\section{Introduction}

In this work we prove the existence of a global bifurcation of radial solutions for the problems
\begin{equation} \label{mc}
- div \left ( \frac{\grad u}{\sqrt{1 + |\grad u|^2}}   \right ) = \lambda\, h(x)\, u,\;\;\;\;\;\; x \in \mathbb{R}^3,
\end{equation}
where, $h(x) = (1+|x|^2)^{-2}$. The branches of solutions bifurcate from the principal eigenvalue $\lambda_0$ of the corresponding linear equation
\begin{equation}\label{eq1.2}
-\Delta u = \lambda\, h(x)\, u,\;\;\; x \in \mathbb{R}^3.
\end{equation}
There is an extensive number of works studying such problems; we refer to the works \cite{b09, cc018, drt18, hprx18, p18, ry20} and the references therein. Equations of type (\ref{mc}) are non uniformly elliptic equations. Due to this, most variational methods fail. One of the main tools to study such problems is bifurcation theory, see \cite{cc018, drt18, mxh20, hprx18, nak90, ry20}. However, in the case where the equations are defined on the whole space, the only work, up to our knowledge, obtained (local) bifurcation results was \cite{sz02}. Our procedure here applies to the general class of equations (see Section \ref{secmc}):
\begin{equation} \label{mcgeneral}
- div \left ( \frac{\grad u}{f(x,u,\nabla u)}   \right ) = \lambda\, h(x)\, u + g(x,u,\nabla u),\;\;\;\;\;\; x \in \mathbb{R}^3.
\end{equation}
In order to state the main result of this paper we give the
following definition of what we mean with the existence of \emph{a
branch of solutions} of an operator equation (\cite{ra71}).
\begin{theorem} \label{rab}
Assume that $X$ is a Banach space with norm\ $||\cdot||$ and
consider\ $G(\lambda,\cdot)= L(\lambda,\cdot) + H(\lambda,\cdot)$,\
where\ $L$\ is a compact linear map on\ $X$\ and\
$H(\lambda,\cdot)$\ is compact and satisfies
\[
\lim_{||u|| \to 0} \frac{||H(\lambda,u)||}{||u||}=0.
\]
If\ $\lambda$\ is a simple eigenvalue of\ $L$\ then the closure of
the set
\[
C=\{ (\lambda ,u) \in \mathbb{R} \times X : (\lambda,u)\;\;
\mbox{solves}\;\; u=G(\lambda,u),\; u \not\equiv 0 \},
\]
possesses a maximal continuum (i.e. connected branch) of
solutions,\ $\mathcal{C}_\lambda$,\ such that\ $(\lambda,0) \in
\mathcal{C}_{\lambda}$\ and\ $\mathcal{C}_{\lambda}$\ either:

(i)\ meets infinity in\ $\mathbb{R} \times X$\ or,

(ii)\ meets\ $(\lambda^*,0)$,\ where\ $\lambda^* \ne \lambda$\ is
also an eigenvalue of\ $L$.
\end{theorem}

One possible way to study problems of the form (\ref{mc}), is to approximate them by suitably elliptic problems, for which the classical methods apply. Our procedure here is based on the arguments of \cite{zog}. \vspace{0.2cm}

We introduce the Hilbert space $X$ to be defined as the completion of $C_0^{\infty} (\mathbb{R}^3)$ radial functions, with norm:
\[
||u||^2_{X} = \sum_{1 \leq |\alpha| \leq 6} ||D^{\alpha} u||^2_{L^2 (\mathbb{R}^3)} + \int_{\mathbb{R}^3} |\grad u|^2\, dx,
\]
i.e.,
\[
X = \left \{ u \in D_r^{1,2} (\mathbb{R}^3);\;\;\; D^{\alpha} u \in H_r^{5} (\mathbb{R}^3),\;\;\; |\alpha|=1 \right \}.
\]
We note that $H_r^{5} (\mathbb{R}^3)$ denotes the radial subspace of the classical Sobolev space $H^{5}$ and $D_r^{1,2} (\mathbb{R}^3)$ is the radial subspace of $D^{1,2} (\mathbb{R}^3)$ for which we refer to Section \ref{seceigen}. Since $X \subset D^{1,2} (\mathbb{R}^3) \subset L^6 (\mathbb{R}^3)$, the space is well defined i.e., the constant function does not belong in $X$. Moreover, $H^{5}$ is continuously imbedded in $C^3 (\mathbb{R}^3)$, thus $X$ is continuously imbedded in $C^4 (\mathbb{R}^3)$. We remind that for radial functions $H_r^1 (\mathbb{R}^3)$ is compactly imbedded in $L^p (\mathbb{R}^3)$, for $2<p<6$. Denote by $<.,.>_X$ the inner product in $X$.

The weak formulation of (\ref{mc}) is:
\[
\int_{\mathbb{R}^3} \nabla u \cdot \nabla \phi\, dx = \lambda\, \int_{\mathbb{R}^3} h\, u\, \phi\, dx + \int_{\mathbb{R}^3} H(\lambda,u)\, \phi\, dx,
\]
for $\phi \in X$ and
\[
H(\lambda,u) = \lambda\, h\, \frac{u\, |\nabla u|^2}{1+\sqrt{1 + |\nabla u|^2}} - \frac{1}{2}\, \frac{\nabla |\nabla u|^2 \cdot \nabla u}{1 + |\nabla u|^2}.
\]
First we will prove the existence of a global branch of solutions for the problem
\begin{equation} \label{eq1.1}
\int_{\mathbb{R}^3} \nabla u \cdot \nabla \phi\, dx = \lambda\, \int_{\mathbb{R}^3} h\, u\, \phi\, dx + \int_{\mathbb{R}^3} H(\lambda,u)\, \frac{u^2}{u^2 + \theta}\, \phi\, dx,
\end{equation}
where $\theta$ is a fixed positive real number. This will be done by assuming, for $\varepsilon$ small enough positive number, the problem:
\begin{equation} \label{eq1.3}
\varepsilon\, <u,\varphi>_X + \int_{\mathbb{R}^3} \nabla u \cdot \nabla \phi\, dx = \lambda\, \int_{\mathbb{R}^3} h\, u\, \phi\, dx + \int_{\mathbb{R}^3} H(\lambda,u)\, \frac{u^2}{u^2 + \theta}\, \phi\, dx.
\end{equation}
The corresponding linear eigenvalue problem of (\ref{eq1.3}) is:
\begin{equation} \label{eq1.4}
\varepsilon\, <u,\varphi>_X + \int_{\mathbb{R}^3} \nabla u \cdot \nabla \phi\, dx = \lambda\, \int_{\mathbb{R}^3} h\, u\, \phi\, dx,\;\;\;\;\;\; \mbox{for any}\;\;\; \phi \in X.
\end{equation}

Our intention is to prove that for each $\varepsilon$ small enough, there exists a branch of solutions of (\ref{eq1.3})
bifurcating from a certain simple eigenvalue of (\ref{eq1.4}), such that these branches converge to the global branch of
(\ref{eq1.1}), as $\varepsilon \downarrow 0$. In this direction, we face two main difficulties. The first is, that the properties of the eigenvalues of (\ref{eq1.4}) are not
known; it is unclear if they are close enough to $\lambda_0$,  and what is their multiplicity. The second difficulty, since we have a
singular perturbation, is to prove that the branches of solutions of (\ref{eq1.3}) converge in $X$; we cannot expect that the branches of solutions of (\ref{eq1.3}) converge
uniformly. However, a careful application of Whyburn's Lemma, is sufficient to give the desired (nonuniform) convergence. Finally, we will take the limit as $\theta$ tends to zero, to obtain the result for (\ref{mc}). \vspace{0.2cm}

In Section 2, we deal with the eigenvalue problem (\ref{eq1.4}). We
prove that for any $\varepsilon>0$ small enough, admits a positive
eigenvalue $\lambda_{0,\varepsilon}$, which is simple with the associated
eigenfunction being positive. The proof is not straightforward. Problem (\ref{eq1.2}) (and (\ref{eq1.5} below),
admits principal eigenvalue which is simple in the sense that the null space of the corresponding operator, $P_0$, is spanned by the
eigenfunction and the range of $P_0$ is the orthogonal complement of the null space in $D^{1,2} (\mathbb{R}^3)$. In other words, $P_0$ is Fredholm with
index 0, or equivalently $\lambda_0$ has algebraic multiplicity one. In our case, we do not have this: considering problem (\ref{eq1.2}) in the space $X$, although $X \subset
D^{1,2} (\mathbb{R}^3)$, the situation is different. In the context of the space $X$, the operator $P_0$ is no longer Fredholm with index 0. This means that the principal eigenvalue
$\lambda_0$ has multiplicity one (since there exists only one function satisfying (\ref{eq1.2})) but has not algebraic multiplicity one. This is the reason why the Theorem of perturbed eigenvalues (see for instance \cite{bt03, Hans04}) cannot be applied in the case of (\ref{eq1.2}) and (\ref{eq1.4}). However, the eigenvalues of the
singular perturbations, that remain close enough to $\lambda_0$, are proved to be algebraic simple.
\begin{theorem} \label{seigen}
There exists $s >0$ small enough, such that for every $\varepsilon$, with $0 < \varepsilon < s$, problem (\ref{eq1.4}) admits an eigenvalue
$\lambda_{0,\varepsilon}$ which is (algebraic) simple in $X$ and the associated normalized eigenfunction $u_{0,\varepsilon} \in X$
is positive. The perturbed eigenpairs form a continuous curve:
\[
\varphi(\varepsilon) = \{(\lambda_{0,\varepsilon},
u_{0,\varepsilon}),\;\;\; 0 \leq \varepsilon < s \} \subset
\mathbb{R} \times X.
\]
Moreover, for $0 \leq \varepsilon < s$, (\ref{eq1.4}) has no other eigenpair, close enough to $(\lambda_0, u_0)$, than that belonging in $\varphi(\varepsilon)$.
\end{theorem}

Let $g_n$ be a smooth enough radial function, at least $C^{4}(\mathbb{R}^3)$, with $g_n \in L^{3/2}(\mathbb{R}^3) \cap H^4 (\mathbb{R}^3)$ and the derivatives of order up to $4$ belong to $L^{3/2} (\mathbb{R}^3)$. The measure of the positive part $|g^+_n|$ is not zero in $\mathbb{R}^3$. Then, Theorem \ref{seigen} is also applicable in the case of the linear problems
\begin{equation} \label{eqsweight}
\varepsilon\, <u,\varphi>_X + \int_{\mathbb{R}^3} \nabla u \cdot \nabla \phi\, dx = \lambda\, \int_{\mathbb{R}^3} g_n\, u\, \phi\, dx,\;\;\;\;\;\; \mbox{for any}\;\;\; \phi \in X.
\end{equation}
For such $g_n$, there exists a principal eigenvalue $\lambda_{n,0}$ for the problem
\begin{equation} \label{eq1.5}
-\Delta u = \lambda\, g_n(x)\, u,\;\;\; x \in \mathbb{R}^3.
\end{equation}
and the corresponding normalized eigenfunction $u_{n,0}$ belongs in $X$. Then, as in Theorem \ref{seigen}, there
exist a local, to $(\lambda_{n,0}, u_{n,0})$, curve of eigenpairs for the perturbed problems (\ref{eqsweight}). Denote this curve by
\[
\varphi_n(\varepsilon) = \{(\lambda_{n,\varepsilon},
u_{n,\varepsilon}),\;\;\; 0 \leq \varepsilon < s \} \subset
\mathbb{R} \times X.
\]

Our last result concerning eigenvalue problems is the following Proposition. For the proof we
refer to Subsection \ref{prmainprop}.
\begin{proposition} \label{mainprop}
Let $g_n$ and $g_0$ be $C^{4}(\mathbb{R}^3)$ functions, and let $(\lambda_{n,0}, u_{n,0})$, $(\lambda_0, u_0)$ be the principal eigenpairs of (\ref{eq1.5}) with weights
$g_n$, $g_0$, respectively. Then, $u_{n,0} \to u_0$ in $X$ if and only if the set
\[
\Phi := \bigcup_{n > N(\delta)} \bigcup_{0 \leq \varepsilon \leq s_*} \varphi_n(\varepsilon),
\]
is compact in $X \times \mathbb{R}$. The positive real numbers $s_*$ and $\delta$, depend only on (\ref{eq1.5}) with weight $g_0$.
\end{proposition}

Based on these results, we prove the existence of a branch of solutions, $\mathcal{C}_{\varepsilon}$, for the problem (\ref{eq1.3}), for any $\varepsilon$ small enough,
bifurcating from $\lambda_{0,\varepsilon}$. We are not in the position to exclude the second alternative of Theorem \ref{rab}. The standard
argument, based on the maximum principle, which ensures that the branches are global and the solutions belonging in these branches
are positive, cannot be applied. However, we prove that if the second alternative of Theorem \ref{rab} holds, then $\lambda^*_{\varepsilon} \to \infty$,
as $\varepsilon \downarrow 0$. \vspace{0.2cm}

In Section 3, we leave $\varepsilon \downarrow 0$, in order
to obtain a global branch of solutions for the problem
(\ref{eq1.1}).
\begin{theorem} \label{global}
The principal eigenvalue $\lambda_{0} > 0$ of the problem (\ref{eq1.2}) is a bifurcation point of the perturbed problem (\ref{eq1.1}). More precisely, the closure of the set
\[
\mathcal{C}_0=\{ (\lambda ,u) \in \mathbb{R} \times
X : (\lambda,u)\;\; \mbox{solves}\;\;
(\ref{eq1.1}),\; u \not\equiv 0 \},
\]
possesses a maximal continuum (i.e. connected branch) of solutions,\ $\mathcal{C}_0$,\ such that\ $(\lambda_{0},0) \in
\mathcal{C}_0$\ and\ $\mathcal{C}_0$\ is unbounded in $\mathbb{R} \times X$. Moreover, there are
maximal connected subsets $\mathcal{C}^+_0$, $\mathcal{C}^-_0$ of
$\mathcal{C}_0$ containing $(\lambda_0,0)$ in their closure, such
that $u(x) > (<)0$, for every $x \in \mathbb{R}^3$, if $(\lambda ,u) \in
\mathcal{C}^{+(-)}_0$, respectively.
\end{theorem}
This is done with the use of \emph{Whyburn's Lemma} \cite{why58}:
\vspace{0.2cm}

Let $G$ be any infinite collection of point sets. The set of all
points $x$ such that every neighborhood of $x$ contains points of
infinitely many sets of $G$ is called the \emph{superior limit of
$G$ ($\limsup G$)}. The set of all points $y$ such that every
neighborhood of $y$ contains points of all but a finite number of
sets of $G$ is called the \emph{inferior limit of $G$ ($\liminf
G$)}.
\begin{theorem} \label{whyburn}
Let $\{ A_n \}_{n \in \mathbb{N}}$ be a sequence of connected
closed sets such that
\[
\liminf_{n \to \infty} \{ A_n \} \not\equiv \emptyset.
\]
Then, if the set $\bigcup_{n \in \mathbb{N}} A_n$ is relatively
compact, $\limsup_{n \to \infty} \{A_n\}$ is a closed, connected
set.
\end{theorem}
The application of Whyburn's Lemma in bifurcation theory is rather standard; We refer in \cite{eg00} and the
references therein. In our case, we prove that $\mathcal{C}_{\varepsilon} \to \mathcal{C}_{0}$ in $\mathbb{R} \times X \cap
B_R(\lambda_0,0)$, for any $R \in \mathbb{R}$, where $B_R (\lambda_0,0)$ denotes the ball of $\mathbb{R} \times
X$ with center $(\lambda_0,0)$ and radius $R$. This means that the branches $\mathcal{C}_{\varepsilon}$ tend to
$\mathcal{C}_0$ locally in $\mathbb{R} \times X$. \vspace{0.3cm} \\
As it is stated in \cite{dai16} (see also the references therein) the above Lemma contains a gap. However, this gap is removable. In any case we
refer to Whyburn's Lemma in the sense of Theorem \cite[Theorem 2.1]{dai16}. \vspace{0.3cm} \\
Finally in Section \ref{secmc}, we state as a theorem the existence of global branches for the Mean curvature equation.
\begin{theorem} \label{thmgeneralcase}
The principal eigenvalue $\lambda_{0} > 0$ of (\ref{eq1.2}) is a bifurcation point of the problem (\ref{mc}). More precisely, the
closure of the set
\[
\mathcal{C}_0=\{ (\lambda ,u) \in \mathbb{R} \times
X : (\lambda,u)\;\; \mbox{solves}\;\;
(\ref{mc}),\; u \not\equiv 0 \},
\]
possesses a maximal continuum (i.e. connected branch) of solutions,\ $\mathcal{C}_0$,\ such that\ $(\lambda_{0},0) \in
\mathcal{C}_0$\ and\ $\mathcal{C}_0$\ is unbounded in $\mathbb{R} \times X$. Moreover, there are
maximal connected subsets $\mathcal{C}^+_0$, $\mathcal{C}^-_0$ of
$\mathcal{C}_0$ containing $(\lambda_0,0)$ in their closure, such
that $u(x) > (<)0$, for every $x \in \mathbb{R}^3$, if $(\lambda ,u) \in
\mathcal{C}^{+(-)}_0$, respectively.
\end{theorem}
\textbf{Notation}  Throughout this work we will consider the weighted space $L^2_g (\mathbb{R}^3)$, with inner product
\[
<\phi,\psi>_{L^2_g (\mathbb{R}^3)} = \int_{\mathbb{R}^3} g(x)\, \phi\, \psi\, dx,
\]
for any $\phi,\; \psi \in C^{\infty}_0 (\mathbb{R}^3)$, respectively.
%
%
\section{Eigenvalue problems} \label{seceigen}

We remind that the space\ $D^{1,2}(\mathbb{R}^3)$, is defined as the closure of the $C^{\infty}_{0}(\mathbb{R}^3)$- functions with respect to the norm
\[
||u||^2_{D^{1,2}(\mathbb{R}^3)}=\int_{\mathbb{R}^3} {| \nabla u |}^2\, dx.
\]
It can be shown that
\[
D^{1,2}(\mathbb{R}^3)= \left\{u \in L^{6}(\mathbb{R}^3): |\nabla u | \in L^2(\mathbb{R}^3) \right\}.
\]
Equivalently, we may define $D^{1,2}(\mathbb{R}^3)$ as the set of $L^1_{loc} (\mathbb{R}^3)$- functions $f$ with $|\nabla f| \in L^2 (\mathbb{R}^3)$ and $f$ vanishes at infinity. Moreover, there exists\ $K>0$\ such that
\begin{equation} \label{Dcnim}
||u||^2_{D^{1,2}(\mathbb{R}^3)} \geq K\, ||u||^2_{L^6(\mathbb{R}^3)},
\end{equation}
for all $u \in D^{1,2}(\mathbb{R}^3)$. In addition, the imbedding
\begin{equation} \label{Dcmim}
D^{1,2}(\mathbb{R}^3) \hookrightarrow L^2_g (\mathbb{R}^3),
\end{equation}
is compact, for every $g \in L^{3/2} (\mathbb{R}^3)$. \vspace{0.2cm} \\
Standard regularity results imply the following:
\begin{lemma} \label{regularity}
Let $u \in X$, be a solution of the problem
(\ref{eq1.3}). Then, $u$ is at least a $C^{6}(\mathbb{R}^3)$-function, such that $D^{\alpha} u \in H^6 (\mathbb{R}^3)$, $|\alpha|=1$.
\end{lemma}
Another well known result concerns the eigenvalue problems (\ref{eq1.2}) and (\ref{eq1.5}):
\begin{lemma} \label{princeig}
Problems (\ref{eq1.2}) and (\ref{eq1.5}) admit principal eigenvalues $\lambda_0$ and $\lambda_{n,0}$, respectively, given by
\begin{equation}\label{varchar}
\lambda_0 = \inf_{0 \not\equiv u \in C^{\infty}_0 (\mathbb{R}^3)} \frac{\int_{\mathbb{R}^3} |\nabla
u|^2\, dx}{\int_{\mathbb{R}^3} h\, u^2\, dx},\;\;\;\;\;\; \lambda_{n,0} = \inf_{0 \not\equiv u \in C^{\infty}_0 (\mathbb{R}^3)} \frac{\int_{\mathbb{R}^3}|\nabla u|^2\, dx}{\int_{\mathbb{R}^3} g_n(x)\, u^2\, dx},
\end{equation}
which are simple and the corresponding normalized eigenvalues $u_0$ and $u_{n,0}$, respectively, belong to $D^{1,2} (\mathbb{R}^3)$ and they are the
only positive eigenfunctions of (\ref{eq1.2}) and (\ref{eq1.5}). Moreover, $u_0$ and $u_{n,0}$ belong also to $X$.
\end{lemma}
\emph{Proof} The existence and the properties of the eigenvalues are proved in \cite{bro94}. We prove that these eigenvalues belong also in $X$; Observe that $g_n \in C^{4}(\mathbb{R}^3) \cap H^{4} (\mathbb{R}^3)$ and $h \in C^{\infty} \cap H^{4} (\mathbb{R}^3)$. In addition, the derivatives of order up to $4$ for both $h$ and $g_n$ belong to $L^{3/2} (\mathbb{R}^3)$. Then, the derivatives of $h\, u_0$ and $g_n\, u_n$, up to order $4$, belong to $L^2 (\mathbb{R}^3)$. Thus, from (\ref{eq1.2}) and (\ref{eq1.5}) we obtain that $u_0$ and $u_{n,0}$ belong also to $X$.
$\blacksquare$ \vspace{0.2cm} \\
Next we consider the eigenvalue problem
\begin{equation}\label{utilde}
<u,v>_X = \lambda\, <u,v>_{L^2_{g_n}},
\end{equation}
for $u$, $v$ in $X$. We denote the eigenpairs of (\ref{utilde}) as $(\tilde{\lambda}^{(i)}_n, \tilde{u}^{(i)}_n)$, $i=1,2,...$. Standard spectral theory for compact
self-adjoint operators imply the existence of these eigenpairs in $\mathbb{R} \times X$ and state their properties. In the sequel, we assume that $\{ \tilde{u}^{(i)}_n \}$
consists an orthonormal basis of $L^2_{g_n} (\mathbb{R}^3)$.
%
%
%
\subsection{Proof of Theorem \ref{seigen}} \label{prseigen}
We give the proof for the more general case of the problem (\ref{eqsweight}). Throughout, this Subsection we assume that $n$, thus $g_n$, is fixed.
\begin{lemma} \label{lem2.1}
Assume that there exist $(v_{n,\varepsilon},\mu_{n,\varepsilon})$ eigenpairs of (\ref{eqsweight}) in $X$, i.e.
\begin{equation}\label{sim1}
\varepsilon\, <v_{n,\varepsilon}, \phi>_X + <v_{n,\varepsilon}, \phi>_{D^{1,2} (\mathbb{R}^3)} - \mu_{n,\varepsilon} <v_{n,\varepsilon}, \phi>_{L^2_{g_n} (\mathbb{R}^3)}=0,
\end{equation}
for any $\phi \in X$, such that $\mu_{n,\varepsilon} \to \lambda_{n,0}$, as $\varepsilon \to 0$. Then, $v_{n,\varepsilon} \to u_{n,0}$, in $X$.
\end{lemma}
\emph{Proof} Without loss of generality we assume that $v_{n,\varepsilon}$ are normalized in $X$. Observe that $v_{n,\varepsilon} \in D^{1,2} (\mathbb{R}^3)$, so we may decompose it
as
\begin{equation}\label{sim2}
v_{n,\varepsilon} = a_{n,\varepsilon} u_{n,0} + \xi_{n,\varepsilon},\;\;\; a_{n,\varepsilon} = <v_{n,\varepsilon},u_{n,0}>_{D^{1,2} (\mathbb{R}^3)} / ||u_{n,0}||_{D^{1,2} (\mathbb{R}^3)},
\end{equation}
where $\xi_{n,\varepsilon} \perp u_{n,0}$ in $D^{1,2} (\mathbb{R}^3)$, hence $\xi_{n,\varepsilon} \perp u_{n,0}$ also in $L^2_{g_n} (\mathbb{R}^3)$. From (\ref{sim1}), we have that
$\xi_{n,\varepsilon}$ satisfy
\begin{eqnarray}\label{sim3}
\varepsilon\, <\xi_{n,\varepsilon}, \phi>_X + <\xi_{n,\varepsilon}, \phi>_{D^{1,2} (\mathbb{R}^3)} - \mu_{n,\varepsilon} <\xi_{n,\varepsilon}, \phi>_{L^2_{g_n} (\mathbb{R}^3)}= \nonumber
\\
=-\varepsilon\, a_{n,\varepsilon} <u_{n,0}, \phi>_X  + a_{n,\varepsilon} (\mu_{n,\varepsilon}- \lambda_{n,0}) <u_{n,0}, \phi>_{L^2_{g_n} (\mathbb{R}^3)},
\end{eqnarray}
for any $\phi \in X$. Setting $\phi=\xi_{n,\varepsilon}$, we get that
\begin{equation}\label{sim4}
\varepsilon\, ||\xi_{n,\varepsilon}||^2_X + ||\xi_{n,\varepsilon}||^2_{D^{1,2} (\mathbb{R}^3)} - \mu_{n,\varepsilon}\, ||\xi_{n,\varepsilon}||^2_{L^2_{g_n} (\mathbb{R}^3)} = -\varepsilon\,
a_{n,\varepsilon} <u_{n,0}, \xi_{n,\varepsilon}>_{X},
\end{equation}
or
\begin{equation}\label{sim5}
||\xi_{n,\varepsilon}||^2_X + \frac{1}{\varepsilon} \left ( ||\xi_{n,\varepsilon}||^2_{D^{1,2} (\mathbb{R}^3)} - \mu_{n,\varepsilon}\, ||\xi_{n,\varepsilon}||^2_{L^2_{g_n} (\mathbb{R}^3)}
\right ) =-a_{n,\varepsilon} <u_{n,0}, \xi_{n,\varepsilon}>_{X}.
\end{equation}
Observe that $a_{n,\varepsilon}$ is a bounded sequence in $\mathbb{R}$, since
\[
<v_{n,\varepsilon},u_{n,0}>_{D^{1,2} (\mathbb{R}^3)} \leq ||v_{n,\varepsilon}||_{D^{1,2} (\mathbb{R}^3)} ||u_{n,0}||_{D^{1,2} (\mathbb{R}^3)} \leq C\, ||v_{n,\varepsilon}||_{X} ||u_{n,0}||_{X}=C.
\]
Also, $\xi_{n,\varepsilon} = v_{n,\varepsilon} - a_{n,\varepsilon} u_{n,0}$ implies that $\xi_{n,\varepsilon}$ is also bounded in $X$. Moreover, since $\xi_{n,\varepsilon}
\perp u_{n,0}$ in $D^{1,2} (\mathbb{R}^3)$ and $\mu_{n,\varepsilon}$ remain close enough to $\lambda_{n,0}$, we have
\begin{equation}\label{sim5a}
||\xi_{n,\varepsilon}||^2_{D^{1,2} (\mathbb{R}^3)} - \mu_{n,\varepsilon}\, ||\xi_{n,\varepsilon}||^2_{L^2_{g_n} (\mathbb{R}^3)}  \geq (\lambda^{(2)}_{n,0}-\mu_{n,\varepsilon})
||\xi_{n,\varepsilon}||^2_{L^2_{g_n} (\mathbb{R}^3)},
\end{equation}
where $\lambda^{(2)}_{n,0}$ denotes the second eigenvalue of (\ref{eq1.5}). Note that $\xi_{n,\varepsilon}$ converge weakly to some $\xi_*$, in $X$, as $\varepsilon
\downarrow 0$. If $\xi_* \not\equiv 0$, then from (\ref{sim4}) and (\ref{sim5a},) we obtain that
\[
(\lambda^{(2)}_{n,0} -\mu_{n,\varepsilon}) ||\xi_{n,\varepsilon}||^2_{L^2_{g_n} (\mathbb{R}^3)} \leq -\varepsilon\, a_{n,\varepsilon} <u_{n,0}, \xi_{n,\varepsilon}>_{X} \to 0,
\]
as $\varepsilon \downarrow 0$, which is a contradiction. Then $\xi_* \equiv 0$ and in this case, from (\ref{sim5}), we have that
\[
||\xi_{n,\varepsilon}||^2_X \leq -a_{n,\varepsilon} <u_{n,0}, \xi_{n,\varepsilon}>_{X} \to 0.
\]
Thus, $\xi_{n,\varepsilon}$ is strongly convergent to zero in $X$. Finally, we conclude that $v_{n,\varepsilon} \to u_{n,0}$, in $X$. $\blacksquare$ \vspace{0.3cm}

We do the following observation: Assume that $u_{n,0} \equiv \tilde{u}^{i}_n$, where $\tilde{u}^{i}_n$ is an eigenfunction of (\ref{utilde}), associated with some eigenvalue
$\tilde{\lambda}^{i}_n$. We are not in the position to exclude this case, but if this happens can be characterized as exceptional and should hold for special weights $g_n$.
In this case, the curve of eigenpairs $\phi (\varepsilon)$ is given directly by $\phi (\varepsilon)= \{ (\lambda_{n\varepsilon}, u_{n,\varepsilon}) = (\lambda_{0,n} +
\varepsilon\, \tilde{\lambda}^{(i)}_n, \tilde{u}^{(i)}_n) \}$, for every $\varepsilon$. The properties of this curve, are given by Theorem \ref{seigen}, and may be proved by
a similar way (see Remark \ref{remexcept}).  \vspace{0.2cm}

Thus, we will consider the general case; $u_{n,0}$ cannot coincide to any $\tilde{u}^{i}_n$. Assume that
\begin{equation}\label{ltilde}
\tilde{\lambda}^{(i)}_n < \kappa_{n,0} < \tilde{\lambda}^{(i+1)}_n,
\end{equation}
for some fixed $i$ and $\kappa_{n,0}$ denotes the ratio
\begin{equation}\label{k0}
\kappa_{n,0} := \frac{||u_{n,0}||^2_X}{||u_{n,0}||^2_{L^2_{g_n}}} = \frac{1}{||u_{n,0}||^2_{L^2_{g_n}}},
\end{equation}
since $u_{n,0}$ is normalized in $X$. From (\ref{ltilde}) we have that $u_{n,0}$ does not belong to $span \{\tilde{u}^{(1)}_n, ... \tilde{u}^{(i)}_n \}$; if we assume the
opposite then
\[
||u_{n,0}||^2_X = \sum_{j=1}^{i} c^2_j\, ||\tilde{u}^{(j)}_n||^2_X = \sum_{j=1}^{i} \tilde{\lambda}^{(j)}_n\, c^2_j\, ||\tilde{u}^{(j)}_n||^2_{L^2_{g_n}} \leq
\tilde{\lambda}^{(i)}_n \sum_{j=1}^{i} ||c_j\, \tilde{u}^{(j)}_n||^2_{L^2_{g_n}} = \tilde{\lambda}^{(i)}_n ||u_{n,0}||^2_{L^2_{g_n}},
\]
thus $\kappa_{n,0} \leq \tilde{\lambda}^{(i)}_n$, which is a contradiction. From (\ref{ltilde}) we also have that $u_{n,0}$ is not orthogonal to $span \{\tilde{u}^{(1)}_n,
... \tilde{u}^{(i)}_n \}$; if the opposite holds then $\kappa_{n,0} \geq \tilde{\lambda}^{(i+1)}_n$, which is also a contradiction. \vspace{0.2cm}

%
%

We note that there exists functions which are orthogonal to $u_{n,0}$ in $D^{1,2} (\mathbb{R}^3)$, but cannot be orthogonal to $u_{n,0}$
in $X$. The only case that this may happen is when $u_{n,0}$ coincides with an eigenfunction of (\ref{utilde}), i.e., in the exceptional case. To see this, assume the
opposite; all functions orthogonal to $u_{n,0}$ in $D^{1,2} (\mathbb{R}^3)$ are orthogonal also in $X$. Then for some eigenfunction $\tilde{u}^{(i)}_n$ we have that
\[
\tilde{u}^{(i)}_n = a\, u_{n,0} + \zeta,
\]
where $a \ne 0$ and $\zeta$ is orthogonal to $u_{n,0}$ in both $X$ and $D^{1,2} (\mathbb{R}^3)$. From (\ref{utilde}) we have that
\[
a\, <u_{n,0},\phi>_X + <\zeta,\phi>_X - \tilde{\lambda}^{(i)}_n\, a\, <u_{n,0},\phi>_{L^2_{g_n}} - \tilde{\lambda}^{(i)}_n\, <\zeta,\phi>_{L^2_{g_n}}=0,
\]
for any $\phi \in X$. Setting now $\phi = u_{n,0}$, we obtain that $u_{n,0}$ must satisfy
\[
|||u_{n,0}||^2_X = \tilde{\lambda}^{(i)}_n\, |||u_{n,0}||^2_{L^2_{g_n}},
\]
which implies $u_{n,0} \equiv \tilde{u}^{(i)}_n$. From the point of view of (\ref{ltilde}) this is a contradiction. \vspace{0.3cm} \\
%
%
%
%
\emph{Proof of Theorem \ref{seigen}}\ We give the proof for the general case where (\ref{ltilde}) holds. Assume that $n$ and $\varepsilon$ are fixed, such that $\varepsilon$
is small enough. We proceed with the proof in four steps. \vspace{0.2cm} \\
\emph{Existence} For every $\varepsilon >0$ small enough, assume the problem (\ref{eq1.4}). The compactness properties of the space $X$ imply, that is a well defined eigenvalue problem, and the set of the eigenvalues consists an orthonormal basis of $L^2_{g_n} (\mathbb{R}^3)$. Moreover, its principal eigenvalue satisfies the variational characterization
\[
\lambda_{n,\varepsilon} = \inf_{0 \not\equiv u \in X} \frac{\varepsilon ||u||^2_X + ||u||^2_{D^{1,2}(\mathbb{R}^3)}}{||u||^2_{L^2_{g_n}(\mathbb{R}^3)}}.
\]
Moreover,
\[
\lambda_{n,\varepsilon} \geq \inf_{0 \not\equiv u \in X} \frac{||u||^2_{D^{1,2}(\mathbb{R}^3)}}{||u||^2_{L^2_{g_n}(\mathbb{R}^3)}} = \lambda_{n,0},
\]
and
\[
\frac{\varepsilon ||u_{n,0}||^2_X + ||u_{n,0}||^2_{D^{1,2}(\mathbb{R}^3)}}{||u_{n,0}||^2_{L^2_{g_n}(\mathbb{R}^3)}} \geq \inf_{0 \not\equiv u \in X} \frac{\varepsilon ||u||^2_X + ||u||^2_{D^{1,2}(\mathbb{R}^3)}}{||u||^2_{L^2_{g_n}(\mathbb{R}^3)}} = \lambda_{n,\varepsilon}.
\]
Since $u_{n,0}$ is normalized in $X$, we obtain that
\begin{equation} \label{eigbound}
\frac{\varepsilon}{||u_{n,0}||^2_{L^2_{g_n}(\mathbb{R}^3)}} + \lambda_{n,0} \geq \lambda_{n,\varepsilon} \geq \lambda_{n,0}.
\end{equation}
This implies that $\lambda_{n,\varepsilon}$ is decreasing, and $\lambda_{n,\varepsilon} \to  \lambda_{n,0}$, as $\varepsilon \to 0$. Then, Lemma (\ref{lem2.1}) implies that $u_{n,\varepsilon} \to u_{n,0}$, in $X$, as $\varepsilon \downarrow 0$.

As a conclusion we get that, for any $n$ and any $0 < \varepsilon$ small enough there exists an eigenpair $(\lambda_{n,\varepsilon}, u_{n,\varepsilon})$ of (\ref{eqsweight}), which form as $\varepsilon \downarrow 0$, a set of eigenpairs $\phi_n (\varepsilon)$ with the eigenfunctions been positive and normalized. Moreover, $\lambda_{n,\varepsilon}$ and $u_{n,\varepsilon}$ maybe written as
\begin{equation}\label{ueX}
u_{n,\varepsilon} = \alpha_{n,\varepsilon} u_{n,0} + \eta_{n,\varepsilon},\;\;\;\;\;\; <\eta_{n,\varepsilon}, u_{n,0}>_X =0,
\end{equation}
where
\begin{equation}\label{a}
\alpha_{n,\varepsilon} = <u_{n,\varepsilon}, u_{n,0}>_X \ne 0,
\end{equation}
and
\begin{equation} \label{bab2a}
\lambda_{n,\varepsilon} = \lambda_{n,0} + \varepsilon\, \kappa_{n,\varepsilon},
\end{equation}
for some $\kappa_{n,\varepsilon}$.  \vspace{0.2cm} \\
\emph{Uniqueness} For any $n$ fixed, assume that there are $(\mu_{n,\varepsilon},v_{n,\varepsilon})$ eigenpairs of (\ref{eqsweight}), $\mu_{n,\varepsilon} \to \lambda_{n,0}$,
$\varepsilon \in [0,s)$, for some $s$ small enough. Assume also that $\mu_{n,\varepsilon} \ne \lambda_{n,\varepsilon}$ and that $v_{n,\varepsilon}$ are normalized in $X$.
Lemma (\ref{lem2.1}) implies that $v_{n,\varepsilon} \to u_{n,0}$, in $X$, as $\varepsilon \downarrow 0$. Then,
\[
(\mu_{n,\varepsilon}-\lambda_{n,\varepsilon}) <v^p_{n,\varepsilon}, u_{n,\varepsilon}>_{L^2_{g_n} (\mathbb{R}^3)} =0,
\]
where $(\lambda_{n,\varepsilon},u_{n,\varepsilon})$ is the eigenpair corresponding to the same $\varepsilon$. Thus, $v_{n,\varepsilon}$ and $u_{n,\varepsilon}$ should be
orthogonal in $L^2_{g_n} (\mathbb{R}^3)$ which is a contradiction, since they both converge to $u_{n,0}$ in $X$, thus also in $L^2_{g_n} (\mathbb{R}^3)$. As a conclusion we have the
uniqueness of $\phi_n (\varepsilon)$. \vspace{0.2cm} \\
\emph{Simplicity} For any $n$ fixed, assume that $(\lambda_{n,\varepsilon},v_{n,\varepsilon})$ is an eigenpair of (\ref{eqsweight}). We assume that $v_{n,\varepsilon}$ is normalized in $X$. Then, Lemma (\ref{lem2.1}) implies that $v_{n,\varepsilon} \to u_{n,0}$, in $X$, as
$\varepsilon \downarrow 0$. This is also a contradiction, since $v_{n,\varepsilon}$ and $u_{n,\varepsilon}$ should be orthogonal in $L^2_{g_n} (\mathbb{R}^3)$. \vspace{0.2cm} \\
\emph{Algebraic Simplicity} Consider the operator $L_{n,\varepsilon} (\varepsilon, u): (0,s) \times X \to X$, defined by
\[
<L_{n,\varepsilon} (\varepsilon, u), \varphi>_X = \varepsilon\, <u, \phi>_X + <u, \phi>_{D^{1,2} (\mathbb{R}^3)} - \lambda_{n,\varepsilon} <u, \phi>_{L^2_{g_n} (\mathbb{R}^3)},
\]
for any $\phi \in X$. From the above step (simplicity) we have that the null space of $L_{n,\varepsilon}$ is $[u_{n,\varepsilon}]$. It suffices to prove that the range of
$L_{n,\varepsilon}$ is $E_{n,\varepsilon}$, the orthogonal complement of ${u_{n,\varepsilon}}$ in $X$. Assume the opposite; Let $v_{n,\varepsilon} \in E_{n,\varepsilon}$,
such that $L_{n,\varepsilon} (\varepsilon, v_{n,\varepsilon})=u_{n,\varepsilon}$, i.e.,
\begin{equation}
\varepsilon\, <v_{n,\varepsilon}, \phi>_X + <v_{n,\varepsilon}, \phi>_{D^{1,2} (\mathbb{R}^3)} - \lambda_{n,\varepsilon} <v_{n,\varepsilon}, \phi>_{L^2_{g_n} (\mathbb{R}^3)} =
<u_{n,\varepsilon}, \phi>_X,
\end{equation}
for any $\phi \in V$. However, this is a contradiction if we set $\phi = u_{n,\varepsilon}$. Thus the proof is completed. $\blacksquare$
\begin{lemma} \label{phicontinuous}
For any fixed $n$, the curve of eigenpairs $\phi_n$ is continuous in $\mathbb{R} \times V$.
\end{lemma}
\emph{Proof} Let $\varepsilon \to \varepsilon_*$. If $\varepsilon_* =0$, from (\ref{eigbound}), we get that $\lambda_{n,\varepsilon} \to \lambda_{n,0}$ and from
Lemma \ref{lem2.1} we have that $u_{n,\varepsilon} \to u_{n,0}$ in $X$. Assume that $\varepsilon_* \ne 0$. Let $\lambda_{n,\varepsilon} \to \lambda_{*}$ in $\mathbb{R}$ and
$u_{n,\varepsilon} \to u_*$ in $X$. Taking the limit, as $\varepsilon \to \varepsilon_*$, in
\begin{equation}\label{bab4a}
\varepsilon <u_{n,\varepsilon}, \phi>_X + <u_{n,\varepsilon}, \phi>_{D^{1,2} (\mathbb{R}^3)} - \lambda_{n,\varepsilon} <u_{n,\varepsilon}, \phi>_{L^2_{g_n}} =0,
\end{equation}
we conclude that $(\lambda_*,u_*)$ is an eigepair corresponding to
$\varepsilon_*$. From Theorem \ref{seigen}, $(\lambda_*,u_*) \in \phi_n$ and the proof is completed. $\blacksquare$
\begin{remark} \label{remexcept}
Assume the exceptional case where $u_{n,0} \equiv \tilde{u}^{i}_n$, where $\tilde{u}^{i}_n$ is an eigenfunction of (\ref{utilde}), associated with some eigenvalue
$\tilde{\lambda}^{i}_n$. Then, the curve of eigenfunctions $\phi (\varepsilon)= \{ (\lambda_{n\varepsilon}, u_{n,\varepsilon}) = (\lambda_{0,n} + \varepsilon\,
\tilde{\lambda}^{(i)}_n, \tilde{u}^{(i)}_n) \}$, for every $\varepsilon$, has the properties of Theorem \ref{seigen}. More precisely, uniqueness maybe proved exactly as above
and simplicity is implied from the simplicity of $u_{n,0}$.
\end{remark}
\subsection{Proof of Proposition \ref{mainprop}} \label{prmainprop}
We remind that if $\varphi_n (\varepsilon) = (\lambda_{n,\varepsilon} , u_{n,\varepsilon})$,
then $(\lambda_{n,\varepsilon} , u_{n,\varepsilon})$ satisfies (\ref{bab4a}). For $\varepsilon =0$, $(\lambda_{n,0},u_{n,0})$ are the principal eigenpairs of (\ref{eq1.5})
with weights $g_n$ and $(\lambda_{0},u_{0})$ denotes the principal eigenpair of (\ref{eq1.5}) with weight $g_0$. \vspace{0.2cm} \\
Next result concerns the asymptotic behaviour of $\phi_n$ as $u_{n,0} \to u_0$ in $X$. In this case, $g_n \to g_0$ in $C^{1}(\mathbb{R}^3)$, $\lambda_{n,0} \to \lambda_{0}$ and $\lambda^{(2)}_{n,0} \to \lambda^{(2)}_0$, where $\lambda^{(2)}_{n,0}$ and $\lambda^{(2)}_{0}$ are the second eigenvalues of (\ref{eq1.5}), with weights $g_n$ and $g_0$, respectively. \vspace{0.2cm} \\
\emph{Proof of Proposition \ref{mainprop}} Without loss of generalization, we assume that $u_{n,0}$, $n \geq 1$, cannot coincide to any $\tilde{u}^{i}_n$; if there where a
subsequence of $u_{n,0}$ such that $u_{n,0} = \tilde{u}^{i}_n = u_0$, for some $i$, then the corresponding curves are $\phi (\varepsilon)= \{ (\lambda_{n\varepsilon},
u_{n,\varepsilon}) = (\lambda_{0,n} + \varepsilon\, \tilde{\lambda}^{(i)}_n, \tilde{u}^{(i)}_n) \}$ and the result follows. However, we
do not exclude $u_0$ to be equal with any eigenfunction of (\ref{utilde}). \vspace{0.2cm}

In what follows leaving $n \to \infty$ means that $u_{n,0} \to u_0$ in $X$. Let $\delta$ be a positive number, such that
\begin{equation}\label{delta}
2 \delta < \lambda^{(2)}_{0}-\lambda_{0}.
\end{equation}
The convergence of $(u_{n,0},\lambda_{n,0}) \to (u_{0},\lambda_{0})$, in $X \times \mathbb{R}$, implies that there exists $N(\delta) \in \mathbb{N}$, such that
\[
|\lambda_{n,0} - \lambda_0|<\delta\;\;\; \mbox{and}\;\;\; \left | 1-\frac{||u_{0}||^2_{L^2_{g_0}(\mathbb{R}^3)}}{||u_{n,0}||^2_{L^2_{g_n}(\mathbb{R}^3)}}  \right | < \delta,
\]
for any $n > N(\delta)$. We introduce the following quantity, depending only on $(\lambda_0,u_0)$:
\begin{equation}\label{s*}
s_* := \frac{\lambda^{(2)}_{0}-\lambda_{0} -2 \delta}{1+\delta} ||u_{0}||^2_{L^2_{g_0}(\mathbb{R}^3)},
\end{equation}
Let $n > N(\delta)$ and $\varepsilon \in [0,s_*]$. Then,
\begin{eqnarray}\label{lbound}
\lambda_0 + \delta + \frac{\varepsilon}{||u_{n,0}||^2_{L^2_{g_n}(\mathbb{R}^3)}} &\leq& \lambda_0 + \delta + \frac{\lambda^{(2)}_{0}-\lambda_{0} -2 \delta}{1+\delta}
\frac{||u_{0}||^2_{L^2_{g_0}(\mathbb{R}^3)}}{||u_{n,0}||^2_{L^2_{g_n}(\mathbb{R}^3)}} \nonumber \\ &\leq& \lambda_0 + \delta + \frac{\lambda^{(2)}_{0}-\lambda_{0} -2 \delta}{1+\delta}
(1+\delta) = \lambda^{(2)}_{0}- \delta.
\end{eqnarray}
Next, for any $n$, we decompose $u_{n,\varepsilon}$ as in (\ref{ueX}), (\ref{a}). Decomposition (\ref{ueX}) implies that
\begin{equation}\label{hestimate}
1 = \alpha^2_{n,\varepsilon} + ||\eta_{n,\varepsilon}||_X^2,
\end{equation}
for every $(n,\varepsilon)$, $0 \leq \varepsilon \leq s_*$. Next we decompose $\eta_{n,\varepsilon}$ as
\begin{equation}\label{eta}
\eta_{n,\varepsilon} = \beta_{n,\varepsilon} u_{n,0} + \xi_{n,\varepsilon},\;\;\;\;\;\; <\xi_{n,\varepsilon}, u_{n,0}>_{D^{1,2} (\mathbb{R}^3)} =0,
\end{equation}
where
\begin{equation}\label{b}
\beta_{n,\varepsilon} = \frac{<\eta_{n,\varepsilon}, u_{n,0}>_{D^{1,2} (\mathbb{R}^3)} }{||u_{n,0}||^2_{D^{1,2} (\mathbb{R}^3)}} = \frac{<\eta_{n,\varepsilon},
u_{n,0}>_{L^2_{g_n}(\mathbb{R}^3)}}{||u_{n,0}||^2_{L^2_{g_n}(\mathbb{R}^3)}}.
\end{equation}
Observe that $<\xi_{n,\varepsilon}, u_{n,0}>_{D^{1,2} (\mathbb{R}^3)} =0$ implies also $<\xi_{n,\varepsilon}, u_{n,0}>_{L^2_{g_n}(\mathbb{R}^3)} =0$. Moreover, from (\ref{ueX}) and (\ref{eta}), $u_{n,\varepsilon}$ maybe written as
\begin{equation}\label{u2}
u_{n,\varepsilon} = (\alpha_{n,\varepsilon} + \beta_{n,\varepsilon}) u_{n,0} + \xi_{n,\varepsilon}.
\end{equation}
This equality and (\ref{ueX}) imply that
\[
\alpha_{n,\varepsilon} u_{n,0} + \eta_{n,\varepsilon} = (\alpha_{n,\varepsilon} + \beta_{n,\varepsilon}) u_{n,0} + \xi_{n,\varepsilon},
\]
so $\beta_{n,\varepsilon}$ satisfies also
\begin{equation}\label{b2}
\beta_{n,\varepsilon} = - <u_{n,0},\xi_{n,\varepsilon}>_X.
\end{equation}
Using (\ref{ueX}), (\ref{bab4a}) implies
\begin{eqnarray} \label{eq4.9}
\varepsilon\, \alpha_{n,\varepsilon}\, <u_{n,0}, \phi>_X + \alpha_{n,\varepsilon}\, (\lambda_{n,0} - \lambda_{n,\varepsilon})\, <u_{n,0}, \phi>_{L^2_{g_n}(\mathbb{R}^3)} +
\varepsilon\, <\eta_{n,\varepsilon}, \phi>_X + \nonumber
\\
+ <\eta_{n,\varepsilon}, \phi>_{D^{1,2} (\mathbb{R}^3)}- \lambda_{n,\varepsilon}\, <\eta_{n,\varepsilon}, \phi>_{L^2_{g_n}(\mathbb{R}^3)} =0,
\end{eqnarray}
since
\begin{equation}\label{u0}
<u_{n,0}, \phi>_{D^{1,2} (\mathbb{R}^3)} - \lambda_{n,0}\, <u_{n,0}, \phi>_{L^2_{g_n}(\mathbb{R}^3)} =0,
\end{equation}
for any $\phi \in X$. Setting $\phi = u_{n,0}$ in (\ref{eq4.9}) and using (\ref{ueX}), (\ref{u0}) we obtain that
\[
\varepsilon\, \alpha_{n,\varepsilon}\, = (\lambda_{n,\varepsilon} - \lambda_{n,0})\, \left ( \alpha_{n,\varepsilon} ||u_{n,0}||^2_{L^2_{g_n}(\mathbb{R}^3)} + <\eta_{n,\varepsilon},
u_{n,0}>_{L^2_{g_n}(\mathbb{R}^3)}  \right ),
\]
since $u_{n,0}$ is normalized in $X$. Finally, using (\ref{b}) and setting
\begin{equation}\label{l}
\lambda_{n,\varepsilon} - \lambda_{n,0} = \varepsilon\, \kappa_{n,\varepsilon},
\end{equation}
we conclude that
\begin{equation}\label{k}
\kappa_{n,\varepsilon} = \kappa_{n,0}\, \frac{\alpha_{n,\varepsilon}}{\alpha_{n,\varepsilon} + \beta_{n,\varepsilon}},
\end{equation}
where $\kappa_{n,0}$ is given by (\ref{k0}). Next we prove some estimates for $\kappa_{n,\varepsilon}$, $\alpha_{n,\varepsilon}$ and $\beta_{n,\varepsilon}$. Holds that
\begin{equation}\label{kestimate1}
\kappa_{n,\varepsilon} >0,\;\;\;\; \mbox{for every}\;\;\; (n,\varepsilon),\;\; 0 \leq \varepsilon \leq s_*,
\end{equation}
Assume that $\kappa_{n,\varepsilon} \leq 0$, for some $(n,\varepsilon)$. Then,
\[
||u_{n,\varepsilon}||^2_{D^{1,2} (\mathbb{R}^3)} < \lambda_{n,\varepsilon} ||u_{n,\varepsilon}||^2_{L^2_{g_n} (\mathbb{R}^3)} \leq \lambda_{n,0} ||u_{n,\varepsilon}||^2_{L^2_{g_n}
(\mathbb{R}^3)},
\]
which is a contradiction to the variational characterization (see (\ref{varchar})) of $\lambda_{n,0}$. Thus, $\kappa_{n,\varepsilon}$ remains positive for $\varepsilon \ne
0$. For $\alpha_{n,\varepsilon}$ we have that, for $\varepsilon \ne 0$, $\alpha_{n,\varepsilon} \ne 0$, since $\kappa_{n,\varepsilon} \ne 0$. Moreover, for any fixed $n$,
$\alpha_{n,\varepsilon} \to 1$, as $\varepsilon \to 0$ and from (\ref{a}) and Lemma \ref{phicontinuous} we have that $\alpha_{n,\varepsilon}$ is continuous. Hence,
\begin{equation}\label{aestimate}
0 < \alpha_{n,\varepsilon} \leq 1,\;\;\;\; \mbox{for every}\;\;\; (n,\varepsilon),\;\; 0 \leq \varepsilon \leq s_*.
\end{equation}
Observe now that for any fixed $n$, $\beta_{n,\varepsilon} \to 0$, as $\varepsilon \to 0$. Then, $\kappa_{n,\varepsilon} \to \kappa_{n,0}$, as $\varepsilon \to 0$.

Next, we consider $\beta_{n,\varepsilon}$. First, we prove that, for $\varepsilon \ne 0$, $\beta_{n,\varepsilon} = 0$, if and only if $u_{n,\varepsilon}=u_{n,0}=\tilde{u}^{(i)}_n$. Let $u_{n,\varepsilon}=u_{n,0}=\tilde{u}^{(i)}_n$. From (\ref{u2}) we get that $\xi_{n,\varepsilon} = 0$, which from (\ref{b2}) implies that $\beta_{n,\varepsilon} = 0$. Let $\beta_{n,\varepsilon} =0$. From (\ref{eta}) we have that $\eta_{n,\varepsilon} \equiv \xi_{n,\varepsilon}$. Setting $\phi= \xi_{n,\varepsilon}$ in (\ref{eq4.9}) we get that
\begin{equation}\label{x1a}
\varepsilon ||\xi_{n,\varepsilon}||^2_{X} + ||\xi_{n,\varepsilon}||^2_{D^{1,2} (\mathbb{R}^3)} - \lambda_{n,\varepsilon}\, ||\xi_{n,\varepsilon}||^2_{L^2_{g_n} (\mathbb{R}^3)}=0.
\end{equation}
However, the same argument as in Lemma \ref{lem2.1} (see (\ref{sim5a})), imply that
\begin{equation}\label{x1}
||\xi_{n,\varepsilon}||^2_{D^{1,2} (\mathbb{R}^3)} - \lambda_{n,\varepsilon} ||\xi_{n,\varepsilon}||^2_{L^2_{g_n} (\mathbb{R}^3)} \geq ||\xi_{n,\varepsilon}||^2_{D^{1,2} (\mathbb{R}^3)} -
\lambda^{(2)}_{n,0}\, ||\xi_{n,\varepsilon}||^2_{L^2_{g_n} (\mathbb{R}^3)} \geq 0,
\end{equation}
where $\lambda^{(2)}_{n,0}$ is the second eigenvalue of (\ref{eq1.5}). Then, (\ref{x1a}) holds only if $\eta_{n,\varepsilon} \equiv \xi_{n,\varepsilon} \equiv 0$, i.e., $u_{n,\varepsilon}=u_{n,0}=\tilde{u}^{(i)}_n$. Thus, $\beta_{n,\varepsilon} = 0$, for $\varepsilon \ne 0$, is equivalent to the exceptional case $u_{n,\varepsilon}=u_{n,0}=\tilde{u}^{(i)}_n$, which is a contradiction to
our assumption. So, $\beta_{n,\varepsilon} \ne 0$, for every $n > N(\delta)$ and $0 < \varepsilon \leq s_*$.

We prove now that $\beta_{n,\varepsilon} > 0$, for every $n > N(\delta)$ and $0 < \varepsilon \leq s_*$. Using decomposition (\ref{u2}) in (\ref{bab4a}), setting also $\phi=\xi_{n,\varepsilon}$, we obtain
\begin{eqnarray*}
\varepsilon\, <(\alpha_{n,\varepsilon} + \beta_{n,\varepsilon}) u_{n,0} + \xi_{n,\varepsilon}, \xi_{n,\varepsilon}>_X + <(\alpha_{n,\varepsilon} + \beta_{n,\varepsilon})
u_{n,0} + \xi_{n,\varepsilon}, \xi_{n,\varepsilon}>_{D^{1,2} (\mathbb{R}^3)} - \\
-\lambda_{n,\varepsilon}\, <(\alpha_{n,\varepsilon} + \beta_{n,\varepsilon}) u_{n,0} + \xi_{n,\varepsilon}, \xi_{n,\varepsilon}>_{L^2_{g_n}(\mathbb{R}^3)} =0,
\end{eqnarray*}
or
\[
\varepsilon\, (\alpha_{n,\varepsilon} + \beta_{n,\varepsilon}) <u_{n,0}, \xi_{n,\varepsilon}>_X + \varepsilon ||\xi_{n,\varepsilon}||^2_{X} + ||\xi_{n,\varepsilon}||^2_{H_0^1
(\mathbb{R}^3)} - \lambda_{n,\varepsilon}\, ||\xi_{n,\varepsilon}||^2_{L^2_{g_n} (\mathbb{R}^3)}=0.
\]
Using now (\ref{b2}) we get that
\begin{equation}\label{x2}
\varepsilon ||\xi_{n,\varepsilon}||^2_{X} + ||\xi_{n,\varepsilon}||^2_{D^{1,2} (\mathbb{R}^3)} - \lambda_{n,\varepsilon}\, ||\xi_{n,\varepsilon}||^2_{L^2_{g_n} (\mathbb{R}^3)}=
\varepsilon\, \beta_{n,\varepsilon} (\alpha_{n,\varepsilon} + \beta_{n,\varepsilon}).
\end{equation}
Observe that $\alpha_{n,\varepsilon} + \beta_{n,\varepsilon} \ne 0$, since if $\alpha_{n,\varepsilon} +
\beta_{n,\varepsilon} = 0$, (\ref{u2}) imply that $u_{n,\varepsilon} = \xi_{n,\varepsilon}$. This is a contradiction to Lemma \ref{lem2.1}. The positivity of
$\alpha_{n,\varepsilon}$ and $\kappa_{n,\varepsilon}$, see (\ref{k}), imply that $\alpha_{n,\varepsilon} + \beta_{n,\varepsilon}$ is also positive. Then, from (\ref{x1}) and
(\ref{x2}) we conclude that $\beta_{n,\varepsilon}$ is positive, for any fixed $n$ and $\varepsilon$ small enough. Note that for fixed $n$, from (\ref{b}) and Lemma
\ref{phicontinuous} we have that $\beta_{n,\varepsilon}$ is continuous. Hence,
\begin{equation}\label{bestimate}
0 < \beta_{n,\varepsilon}\;\;\; \mbox{for}\;\; \varepsilon \ne 0\;\;\; \mbox{and}\;\;\; \beta_{n,0} = 0,
\end{equation}
for every $n > N(\delta)$ and $0 \leq \varepsilon \leq s_*$. For positive $\beta_{n,\varepsilon}$, we also get that
\begin{equation}\label{kestimate}
\kappa_{n,\varepsilon} \leq ||u_{n,0}||^{-2}_{L^2_{g_n} (\mathbb{R}^3)},
\end{equation}
and finally, from (\ref{l}), (\ref{kestimate}) and (\ref{lbound}) we obtain the uniform bound for $\lambda_{n,\varepsilon}$
\begin{equation}\label{lfinal}
\lambda_{n,\varepsilon} = \lambda_{n,0} + \varepsilon\, \kappa_{n,\varepsilon} \leq \lambda_{0} + \delta + \varepsilon\, ||u_{n,0}||^{-2}_{L^2_{g_n} (\mathbb{R}^3)} \leq \lambda^{(2)}_0 - \delta,
\end{equation}
for every $n > N(\delta)$ and $0 \leq \varepsilon \leq s_*$. Our last estimate concerns $\eta_{n,\varepsilon}$ in terms of $\alpha_{n,\varepsilon}$ and
$\beta_{n,\varepsilon}$.
Setting $\phi=\eta_{n,\varepsilon}$ in (\ref{eq4.9}), we have
\[
\varepsilon ||\eta_{n,\varepsilon}||^2_{X} + ||\eta_{n,\varepsilon}||^2_{D^{1,2} (\mathbb{R}^3)} - \lambda_{n,\varepsilon}\, ||\eta_{n,\varepsilon}||^2_{L^2_{g_n} (\mathbb{R}^3)}=
\alpha_{n,\varepsilon}\, (\lambda_{n,\varepsilon} - \lambda_{n,0})\, <u_{n,0}, \eta_{n,\varepsilon}>_{L^2_{g_n}(\mathbb{R}^3)}.
\]
Using (\ref{l}), (\ref{k}) and (\ref{b}) we get that
\begin{equation}\label{h1}
\varepsilon ||\eta_{n,\varepsilon}||^2_{X} + ||\eta_{n,\varepsilon}||^2_{D^{1,2} (\mathbb{R}^3)} - \lambda_{n,\varepsilon}\, ||\eta_{n,\varepsilon}||^2_{L^2_{g_n} (\mathbb{R}^3)}=
\varepsilon\, \frac{\alpha^2_{n,\varepsilon}\, \beta_{n,\varepsilon}}{\alpha_{n,\varepsilon} + \beta_{n,\varepsilon}}.
\end{equation}
Using decomposition (\ref{eta}) and (\ref{l})-(\ref{k}) we have that
\begin{eqnarray*}
&&||\eta_{n,\varepsilon}||^2_{D^{1,2} (\mathbb{R}^3)} - \lambda_{n,\varepsilon}\, ||\eta_{n,\varepsilon}||^2_{L^2_{g_n} (\mathbb{R}^3)} = \\
&&=  \beta^2_{n,\varepsilon}\, ||u_{n,0}||^2_{D^{1,2} (\mathbb{R}^3)} - \beta^2_{n,\varepsilon}\, \lambda_{n,\varepsilon}\, ||u_{n,0}||^2_{L^2_{g_n} (\mathbb{R}^3)} +
||\xi_{n,\varepsilon}||^2_{D^{1,2} (\mathbb{R}^3)} - \lambda_{n,\varepsilon}\, ||\xi_{n,\varepsilon}||^2_{L^2_{g_n} (\mathbb{R}^3)} =\\
&&=  -\beta^2_{n,\varepsilon}\, (\lambda_{n,\varepsilon} - \lambda_{n,0})\, ||u_{n,0}||^2_{L^2_{g_n} (\mathbb{R}^3)} + ||\xi_{n,\varepsilon}||^2_{D^{1,2} (\mathbb{R}^3)} -
\lambda_{n,\varepsilon}\, ||\xi_{n,\varepsilon}||^2_{L^2_{g_n} (\mathbb{R}^3)}= \\
&&=  - \varepsilon\, \frac{\beta^2_{n,\varepsilon}\, \alpha_{n,\varepsilon}}{\alpha_{n,\varepsilon} + \beta_{n,\varepsilon}} + ||\xi_{n,\varepsilon}||^2_{D^{1,2} (\mathbb{R}^3)} -
\lambda_{n,\varepsilon}\, ||\xi_{n,\varepsilon}||^2_{L^2_{g_n} (\mathbb{R}^3)}.
\end{eqnarray*}
Then (\ref{h1}) is written as
\[
\varepsilon ||\eta_{n,\varepsilon}||^2_{X} - \varepsilon\, \frac{\beta^2_{n,\varepsilon}\, \alpha_{n,\varepsilon}}{\alpha_{n,\varepsilon} + \beta_{n,\varepsilon}} +
||\xi_{n,\varepsilon}||^2_{D^{1,2} (\mathbb{R}^3)} - \lambda_{n,\varepsilon}\, ||\xi_{n,\varepsilon}||^2_{L^2_{g_n} (\mathbb{R}^3)} = \varepsilon\, \frac{\alpha^2_{n,\varepsilon}\,
\beta_{n,\varepsilon}}{\alpha_{n,\varepsilon} + \beta_{n,\varepsilon}},
\]
or
\begin{equation} \label{eq4.10}
\varepsilon ||\eta_{n,\varepsilon}||^2_{X} + ||\xi_{n,\varepsilon}||^2_{D^{1,2} (\mathbb{R}^3)} - \lambda_{n,\varepsilon}\, ||\xi_{n,\varepsilon}||^2_{L^2_{g_n} (\mathbb{R}^3)} =
\varepsilon\, \alpha_{n,\varepsilon}\, \beta_{n,\varepsilon}.
\end{equation}
Observe now that $\xi_{n,\varepsilon}$ is orthogonal to $u_{n,0}$ in $D^{1,2} (\mathbb{R}^3)$ and that the estimate (\ref{lfinal}) is uniform, i.e., holds for any $n > N(\delta)$ and
$0 \leq \varepsilon \leq s_*$, thus
\[
||\xi_{n,\varepsilon}||^2_{D^{1,2} (\mathbb{R}^3)} - \lambda_{n,\varepsilon}\, ||\xi_{n,\varepsilon}||^2_{L^2_{g_n} (\mathbb{R}^3)} >0.
\]
Finally, from (\ref{eq4.10}) we conclude that
\begin{equation}\label{hfinal}
||\eta_{n,\varepsilon}||^2_{X} < \alpha_{n,\varepsilon}\, \beta_{n,\varepsilon},
\end{equation}
for any $n > N(\delta)$ and $0 \leq \varepsilon \leq s_*$. \vspace{0.2cm}

The final step is to prove that every sequence $(u_{n,\varepsilon}, \lambda_{n,\varepsilon}) \in \phi_n$, converges (up to some subsequence) to some $(u_{0,\varepsilon},
\lambda_{0,\varepsilon}) \in \phi_0$, for every $n > N(\delta)$ and every $0 \leq \varepsilon \leq s_*$, in $X \times \mathbb{R}$. Let $(u_{n,\varepsilon},
\lambda_{n,\varepsilon}) \in \phi_n$. Holds that $u_{n,\varepsilon}$ are normalized, hence bounded and $\lambda_{n,\varepsilon}$ is also bounded (see (\ref{lfinal})). Then,
up to some subsequence, $u_{n,\varepsilon} \rightharpoonup u_*$ in $X$ and $\lambda_{n,\varepsilon} \to \lambda_*$.

Assume that $\varepsilon \to \varepsilon_*$, for some $\varepsilon_* \ne 0$. From (\ref{bab4a}) we have that
\[
\varepsilon ||u_{n,\varepsilon}||^2_X + ||u_{n,\varepsilon}||^2_{D^{1,2} (\mathbb{R}^3)} - \lambda_{n,\varepsilon}\, ||u_{n,\varepsilon}||^2_{L^2_{g_n}} = 0.
\]
or
\[
\varepsilon + ||u_{n,\varepsilon}||^2_{D^{1,2} (\mathbb{R}^3)} - \lambda_{n,\varepsilon}\, ||u_{n,\varepsilon}||^2_{L^2_{g_n}} = 0,
\]
since $u_{n,\varepsilon}$ are normalized in $X$. Taking the limit $n \to \infty$ and $\varepsilon \to \varepsilon_*$, we get that
\[
\varepsilon_* + ||u_*||^2_{D^{1,2} (\mathbb{R}^3)} - \lambda_*\, ||u_*||^2_{L^2_{g_n}} = 0.
\]
Thus, $(u_*,\lambda_*) \ne (0,0)$. Moreover, (\ref{bab4a}) implies that
\[
\varepsilon_* <u_*,\phi>_X + <u_*,\phi>_{D^{1,2} (\mathbb{R}^3)} - \lambda_*\, <u_*,\phi>_{L^2_{g_n}} = 0,
\]
for any $\phi \in X$, thus $(u_*,\lambda_*)$ is an eigenpair belonging to $\phi_0$ and that $(u_{n,\varepsilon}, \lambda_{n,\varepsilon}) \to (u_*,\lambda_*)$ in $X \times
\mathbb{R}$.

Assume now that $\varepsilon_* =0$. We will prove that $(u_{n,\varepsilon}, \lambda_{n,\varepsilon}) \to (u_0,\lambda_0)$ in $X \times \mathbb{R}$. Assume that $u_* \equiv
0$, then
\[
\alpha_{n,\varepsilon} = <u_{n,\varepsilon}, u_{n,0}> \to 0,
\]
and (\ref{hfinal}) implies that also
\[
||\eta_{n,\varepsilon}||^2_{X} \to 0.
\]
However, these two limits contradict (\ref{hestimate}). Thus, $u_* \not\equiv 0$ and follows directly that $(u_*, \lambda_*) \equiv (u_0,\lambda_0)$ and since both are
normalized in $X$ we deduce that $u_{n,\varepsilon} \to u_0$, in $X$.

Finally we conclude that, for every $n > N(\delta)$ and $0 \leq \varepsilon \leq s_*$, every sequence $(u_{n,\varepsilon}, \lambda_{n,\varepsilon}) \in \phi_n$, has a convergent subsequence to some $(u_{0,\varepsilon}, \lambda_{0,\varepsilon}) \in \phi_0$, in $X \times \mathbb{R}$, i.e., the set $\Phi$ is compact in $X \times \mathbb{R}$.  \vspace{0.2cm}

If now $\Phi$ is compact in $X \times \mathbb{R}$, then it is immediate that $(u_{n,0},\lambda_{n,0})$ tends to $(u_{0},\lambda_{0})$ in $X \times \mathbb{R}$ and the proof is completed. $\blacksquare$ \vspace{0.2cm}

%
%
\section{Bifurcation Results}
In this section, we will prove the existence of a global branch of solutions for the problem (\ref{eq1.1})
bifurcating from the principal eigenvalue of the linear problem (\ref{eq1.2}). To do this, we will
first prove the existence of global branches for the problems (\ref{eq1.3}) bifurcating from eigenvalues of
the linear problem (\ref{eq1.4}). Then we will leave $\varepsilon \to 0$. \vspace{0.2cm} \\
From Theorem \ref{seigen} we have the existence of a continuous curve,
$\phi_{0} (\varepsilon) = (u_{0,\varepsilon}, \lambda_{0,\varepsilon})$, of eigenpairs for the corresponding perturbed problems (\ref{eq1.4}).
\begin{theorem} \label{eglobal} Let $\varepsilon$ be a
positive real number which is small enough. Then, for each $\varepsilon$, the
eigenvalue $\lambda_{0,\varepsilon}$ of the problem
(\ref{eq1.4}) is a bifurcation point of the perturbed problem
(\ref{eq1.3}). More precisely, the closure of the set
\[
C=\{ (\lambda ,u) \in \mathbb{R} \times X :
(\lambda,u)\;\; \mbox{solves}\;\; (\ref{eq1.3}),\;
u \not\equiv 0 \},
\]
possesses a maximal continuum (i.e. connected branch) of
solutions,\ $\mathcal{C}_{\varepsilon}$,\ such that\
$(\lambda_{0,\varepsilon},0) \in \mathcal{C}_{\varepsilon}$\ and\
$\mathcal{C}_{\varepsilon}$\ either:

(i)\ meets infinity in\ $\mathbb{R} \times X$\ or,

(ii)\ meets\ $(\lambda^*,0)$,\ where\ $\lambda^* \ne \lambda$\ is
also an eigenvalue of\ $L$.
\end{theorem}
\emph{Proof} The proof follows directly from Rabinowitz's Theorem
\ref{rab}. We sketch the proof. For any fixed $\varepsilon>0$, small enough, we consider the operator $G: X \to X$, as
\[
<G(\lambda,u),\phi>_X = \frac{1}{\varepsilon}\, \left ( \int_{\mathbb{R}^3} \nabla u \cdot \nabla \phi\, dx - \lambda\, \int_{\mathbb{R}^3} h\, u\, \phi\, dx - \int_{\mathbb{R}^3} H(\lambda,u)\, \frac{u^2\, \phi}{u^2 +\theta}\, dx\right ).
\]
Observe that $G$ may be written as $G(\lambda,u) = \frac{1}{\varepsilon} L(\lambda,u) + \frac{1}{\varepsilon} \mathcal{H}(\lambda,u)$,
where $L$ is the linear and compact operator
\[
<L(\lambda,u),\phi>_X = \int_{\mathbb{R}^3} \nabla u \cdot \nabla \phi\, dx - \lambda \int_{\mathbb{R}^3} h\, u\, \phi\, dx,
\]
while $\mathcal{H}$ is the nonlinear operator
\[
<\mathcal{H}(\lambda,u),\phi>_X = \int_{\mathbb{R}^3} H(\lambda,u)\, \frac{u^2\, \phi}{u^2 + \theta}\, dx,
\]
satisfying
\begin{equation} \label{H}
<\mathcal{H}(\lambda,u),v>_X \leq C\, ||u||_{X} \sum_{|\alpha|=1} ||D^{\alpha}u||_{L^{12/5} (\mathbb{R}^3)}\, ||D^{\alpha}u||_{L^{12/5} (\mathbb{R}^3)}\, ||\phi||_{L^6 (\mathbb{R}^3)}.
\end{equation}
The first derivatives of $u$ belong in $H^1 (\mathbb{R}^3)$ which implies that $H$ is compact, such that $\mathcal{H}(\lambda,u) = o(||u||_X)$, as $||u||_X \to 0$, for any  $v \in X$. In addition, from Theorem \ref{seigen}, $\lambda_{0,\varepsilon}$ is a simple eigenvalue of $\frac{1}{\varepsilon} L$.
It is clear that the conditions of Rabinowitz's Theorem are satisfied and the result follows.\ $\blacksquare$. \vspace{0.2cm}
%
%
%
%
%

Next we prove some results, concerning the asymptotic behavior of the branches $C_{\varepsilon}$. Assume a sequence $(\lambda_{\varepsilon_n}, u_{\varepsilon_n}) \in
\mathcal{C}_{\varepsilon_n}$, and set
\begin{equation}\label{gbif}
g_{\varepsilon_n} (x) = h + \frac{1}{\lambda_{\varepsilon_n}}\, H(\lambda_{\varepsilon_n},u_{\varepsilon_n})\, \frac{u_{\varepsilon_n}}{u^2_{\varepsilon_n} + \theta},
\end{equation}
We consider the following eigenvalue problems with weight $g_{\varepsilon_n}$:
\begin{equation}\label{eq3.3}
\int_{\mathbb{R}^3} \nabla v_{n} \cdot \nabla \psi\, dx = \lambda_{n}\, \int_{\mathbb{R}^3} g_{\varepsilon_n}\, v_{n}\, \psi\, dx,
\end{equation}
or
\begin{equation} \label{eq3.3a}
-\Delta v_{n} = \lambda_{n}\, g_{\varepsilon_n}\, v_{n},\;\;\; x \in \mathbb{R}^3.
\end{equation}
From Lemma \ref{regularity} we have that $u_{\varepsilon_n} \in C^{6}(\mathbb{R}^3)$, so $g_{\varepsilon_n} \in C^{4}(\mathbb{R}^3)$. Moreover, $g_{\varepsilon_n} \in L^{3/2}(\mathbb{R}^3) \cap H^4 (\mathbb{R}^3)$ and the derivatives of order up to $4$ belong to $L^{3/2} (\mathbb{R}^3)$. Lemma \ref{princeig} implies the existence of a principal eigenvalue $\lambda_{n}$ with the associated normalized eigenfunction $v_{n}$ being positive and belonging in $X$. We denote by $g_0$ the weight function of (\ref{eq1.2}), i.e., $g_0 \equiv h$.
\begin{lemma} \label{lemgc1}
Let $(\lambda_{\varepsilon_n}, u_{\varepsilon_n}) \in \mathcal{C}_{\varepsilon_n}$ and assume that $(\lambda_{\varepsilon_n}, u_{\varepsilon_n})$ is a bounded sequence in $X$, such that
$u_{\varepsilon_n} \rightharpoonup 0$ in $X$. Then, $(\lambda_n, v_{n}) \to (\lambda_0, u_0)$ in $\mathbb{R} \times X$.
\end{lemma}
\emph{Proof} Let $(\lambda_{\varepsilon_n}, v_n)$ be the principal eigenpairs of (\ref{eq3.3a}). Note that $v_n$ are normalized in $X$, hence $v_n$ is a bounded sequence in $X$, and also in $C^4 (\mathbb{R}^3)$, such that $v_n \rightharpoonup v_*$ in $X$. Moreover, $(\lambda_{\varepsilon_n}, u_{\varepsilon_n})$ satisfy (\ref{eq1.3}) i.e.,
\begin{equation}\label{eq3.1}
\varepsilon_n\, ||u_{\varepsilon_n}||^2_X + \int_{\mathbb{R}^3} |\nabla u_{\varepsilon_n}|^2\, dx = \lambda_{\varepsilon_n}\, \int_{\mathbb{R}^3} h\, u_{\varepsilon_n}^2\, dx + \int_{\mathbb{R}^3} H(\lambda_{\varepsilon_n},u_{\varepsilon_n})\, \frac{u_{\varepsilon_n}^2}{u_{\varepsilon_n}^2 + \theta}\, u_{\varepsilon_n}\, dx.
\end{equation}

We prove first that $\lambda_{\varepsilon_n}$ is bounded away from zero. Assume the opposite, then dividing (\ref{eq3.1}) by $||u_{\varepsilon_n}||^2_{D^{1,2}(\mathbb{R}^3)}$ and setting
\[
\tilde{u}_{\varepsilon_n} = \frac{u_{\varepsilon_n}}{||u_{\varepsilon_n}||_{D^{1,2}(\mathbb{R}^3)}},
\]
we obtain that
\begin{equation} \label{eq2.17}
\varepsilon_n ||\tilde{u}_{\varepsilon_n}||^2_X + ||\tilde{u}_{\varepsilon_n}||^2_{D^{1,2}(\mathbb{R}^3)} = \lambda_{\varepsilon_n}\, \int_{\mathbb{R}^3} h\, \tilde{u}_{\varepsilon_n}^2\, dx + \int_{\mathbb{R}^3} H(\lambda_{\varepsilon_n},u_{\varepsilon_n})\, \frac{u_{\varepsilon_n}}{u_{\varepsilon_n}^2 + \theta}\, \tilde{u}_{\varepsilon_n}^2\, dx.
\end{equation}
Since $\tilde{u}_{\varepsilon_n}$ are normalized in $D^{1,2}(\mathbb{R}^3)$, they converge weakly in $D^{1,2}(\mathbb{R}^3)$ to some function $\tilde{u}_*$. The compact imbedding
$D^{1,2}(\mathbb{R}^3) \hookrightarrow L_{h}^{2}(\mathbb{R}^3)$ implies that $\tilde{u}_{\varepsilon_n} \to \tilde{u}_*$, in $L_{h}^{2}(\mathbb{R}^3)$. Moreover, $u_{\varepsilon_n}$ is a bounded sequence in $X$, thus in $C^4 (\mathbb{R}^3)$. Thus,
\begin{eqnarray} \label{H0}
\int_{\mathbb{R}^3} \frac{H\, u_{\varepsilon_n}}{u_{\varepsilon_n}^2+\theta}\, \tilde{u}_{\varepsilon_n}^2\, dx &=& \frac{1}{||u_{\varepsilon_n}||_{D^{1,2} (\mathbb{R}^3)}} \int_{\mathbb{R}^3} \frac{H\, u_{\varepsilon_n}^2}{u_{\varepsilon_n}^2+\theta}\, \tilde{u}_{\varepsilon_n}\, dx \leq \frac{1}{||u_{\varepsilon_n}||_{D^{1,2} (\mathbb{R}^3)}} \int_{\mathbb{R}^3} |H|\, |\tilde{u}_{\varepsilon_n}|\, dx \nonumber \\
&\leq& C\, \frac{1}{||u_{\varepsilon_n}||_{D^{1,2} (\mathbb{R}^3)}} \left ( \int_{\mathbb{R}^3} |\nabla u_{\varepsilon_n}|^2\, |\tilde{u}_{\varepsilon_n}|\, dx + \sum_{|\alpha|=1} \int_{\mathbb{R}^3} |D^{\alpha} u_{\varepsilon_n}|\, |\nabla u_{\varepsilon_n}|\, |\tilde{u}_{\varepsilon_n}|\, dx \right ) \nonumber \\
&\leq& \left ( ||\nabla u_{\varepsilon_n}||_{L^3 (\mathbb{R}^3)} + \sum_{|\alpha|=1} ||D^{\alpha} u_{\varepsilon_n}||_{L^3 (\mathbb{R}^3)}  \right ) ||\tilde{u}_{\varepsilon_n}||_{L^6 (\mathbb{R}^3)} \to 0,
\end{eqnarray}
since the first derivatives of $u_{\varepsilon_n}$ belong in $H_r^1 (\mathbb{R}^3)$ and the compact embedding $H_r^1 (\mathbb{R}^3)$ to $L^p (\mathbb{R}^3)$, $2<p<6$, hold.
Taking now the limit as $\varepsilon_n \to 0$ in (\ref{eq2.17}) we get
\begin{equation} \label{eq2.19}
\lim_{\varepsilon_n \to \infty} \left ( A_{\varepsilon_n} \right ) + 1 = \lambda_*\, \int_{\mathbb{R}^3} h\, g_{\varepsilon_n}\, |\tilde{u}_*|^2\, dx,
\end{equation}
where
\[
A_{\varepsilon_n} = \varepsilon_n ||\tilde{u}_{\varepsilon_n}||^2_X.
\]
Note that $\lambda_* =0$ and this is a contradiction, so $\lambda_{\varepsilon_n} \nrightarrow 0$.

Next we note that $\lambda_{n}$ cannot also tend to zero; Assume the opposite; from (\ref{eq3.3}) we get that
\begin{eqnarray} \label{eq4h1a}
1 = \int_{\mathbb{R}^3} |\nabla \bar{v}_{n}|^2\, dx = \lambda_{n}\, \int_{\mathbb{R}^3} g_{\varepsilon_n}\, |\bar{v}_{n}|^2\, dx  \nonumber \\ =  \lambda_{n}\, \int_{\mathbb{R}^3} h\, |\bar{v}_{n}|^2\, dx + \int_{\mathbb{R}^3} H(\lambda_{\varepsilon_n},u_{\varepsilon_n})\, \frac{u_{\varepsilon_n}}{u^2_{\varepsilon_n} + \theta}\, |\bar{v}_{n}|^2\, dx,
\end{eqnarray}
where $\bar{v}_n$ denotes the normalization of $v_n$ in $D^{1,2} (\mathbb{R}^3)$. However, the last integral tends to zero, as in (\ref{H0}). Hence (\ref{eq4h1a}) implies a contradiction. Moreover, $\lambda_{n}$ are bounded; The variational characterization of $\lambda_{n}$ implies that
\[
\lambda_0 = \frac{\int_{\mathbb{R}^3} |\nabla u_{0}|^2\, dx}{\int_{\mathbb{R}^3} g_{0}\, |u_0|^2\, dx} = \frac{\int_{\mathbb{R}^3} |\nabla u_{0}|^2\, dx}{\int_{\mathbb{R}^3}  g_{\varepsilon_n}\, |u_0|^2\, dx} \frac{\int_{\mathbb{R}^3} g_{\varepsilon_n}\, |u_0|^2\, dx}{\int_{\mathbb{R}^3} h\, g_{0}\, |u_0|^2\, dx}  \geq \lambda_{n}. \frac{\int_{\mathbb{R}^3} g_{\varepsilon_n}\, |u_0|^2\, dx}{\int_{\mathbb{R}^3} g_{0}\, |u_0|^2\, dx}.
\]
However, see (\ref{H0}),
\begin{equation}\label{eq4h0}
\int_{\mathbb{R}^3} g_{\varepsilon_n}\, |u_0|^2\, dx \to \int_{\mathbb{R}^3} g_{0}\, |u_0|^2\, dx.
\end{equation}
Hence, $\lambda_{n}$ is a bounded sequence and we conclude that $\lambda_{n}$ converges, up to a subsequence, to some $\lambda_* >0$. \vspace{0.2cm}

Assume that $v_* \not\equiv 0$. From (\ref{eq3.3}) we get that
\begin{equation} \label{eq4h1}
0 < \int_{\mathbb{R}^3} |\nabla v_{n}|^2\, dx = \lambda_{n}\, \int_{\mathbb{R}^3} g_{\varepsilon_n}\, |v_{n}|^2\, dx.
\end{equation}
Thus, taking the limit as $n \to \infty$, from (\ref{eq4h1}) we get that
\[
\int_{\mathbb{R}^3} |\nabla v_{*}|^2\, dx = \lambda_{*}\, \int_{\mathbb{R}^3} g_{0}\, |v_{*}|^2\, dx.
\]
Then the positivity of $v_*$, implies that $\lambda_* = \lambda_0$ and $v_* \equiv u_0$. Since $v_n$ and $v_*$ are both normalized in $X$, we conclude that $v_{n} \to u_0$ in $X$. \vspace{0.2cm}

Assume that $v_* \equiv 0$. This means that
\begin{equation} \label{eq4h1c}
||v_{n}||_{D^{1,2} (\mathbb{R}^3)} \to 0.
\end{equation}
Denote by $\bar{v}_n$ the normalization of $v_n$ in $D^{1,2} (\mathbb{R}^3)$. We claim that $\bar{v}_{n} \to v_* \equiv \bar{v}_0$, in $D^{1,2} (\mathbb{R}^3)$, where $\bar{v}_0$ denotes the normalization of $v_0$ in $D^{1,2} (\mathbb{R}^3)$: From (\ref{eq3.3}) we have that
\[
1 = \int_{\mathbb{R}^3} |\nabla \bar{v}_{n}|^2\, dx = \lambda_{n}\, \int_{\mathbb{R}^3} g_{\varepsilon_n}\, |\bar{v}_{n}|^2\, dx,
\]
or
\begin{equation} \label{eq4h1aa}
1 = \int_{\mathbb{R}^3} |\nabla \bar{v}_{n}|^2\, dx =
\lambda_{n}\, \int_{\mathbb{R}^3} g_{0}\, |\bar{v}_{n}|^2\, dx\,
+ \lambda_{n}\,  \int_{\mathbb{R}^3} ( g_{\varepsilon_n}-g_{0} )\, |\bar{v}_{n}|^2\, dx.
\end{equation}
However, from (\ref{H0}) we have that the last integral of (\ref{eq4h1aa}) tends to zero. Moreover, the compact (\ref{Dcmim}) implies that $\bar{v}_{n} \to v_* \equiv \bar{v}_0$, in $D^{1,2} (\mathbb{R}^3)$.

Observe now that $\bar{v}_{n}$ satisfy
\[
- \Delta \bar{v}_{n} = \lambda_{n}\, g_{\varepsilon_n}\, \bar{v}_{n},\;\;\; x \in \mathbb{R}^3.
\]
Since $u_{\varepsilon_n} \in C^4 (\mathbb{R}^3)$ and $D^{\alpha} u_{\varepsilon_n} \in H^1 (\mathbb{R}^3)$, $|\alpha| \leq 5$, we conclude that $\bar{v}_{n}$ is bounded in $X$, which contradicts (\ref{eq4h1c}).

Finally, $(\lambda_n, v_{n}) \to (\lambda_0, u_0)$ in $\mathbb{R} \times X$ and the proof is completed. $\blacksquare$ \vspace{0.2cm}
%
%
%
%
%
%
\begin{lemma} \label{basic argument}
Let $\varepsilon_n > 0$, be a sequence that tends to $0$, as $n \to
\infty$, and $(\lambda_{\varepsilon_n}, u_{\varepsilon_n}) \in \mathcal{C}_{\varepsilon_n}$. Assume
that $(\lambda_{\varepsilon_n}, u_{\varepsilon_n})$ is a bounded sequence in
$\mathbb{R} \times X$, such that $u_{\varepsilon_n}$ being nonnegative. Denote by $\lambda_*$ and $u_*$ the limits (up to a subsequence) $\lambda_{\varepsilon_n} \to
\lambda^*$ and $u_{\varepsilon_n} \rightharpoonup u_*$ (weakly) in $X$. Then, $u_{\varepsilon_n}$ converges strongly in $u_*$ in $X$, and either
\begin{enumerate}
  \item $(\lambda_*, u_*)$ is a nontrivial solution of (\ref{eq1.1}), or
  \item $u_* \equiv 0$. In this case, $\lambda_* = \lambda_0$ and $u_{\varepsilon_n} \to 0$ in $X$, such that the normalization of $u_{\varepsilon_n}$ converges to $u_0$ in
      $X$.
\end{enumerate}
\end{lemma}
\emph{Proof} Since $(\lambda_{\varepsilon_n}, u_{\varepsilon_n}) \in \mathcal{C}_{\varepsilon_n}$, they satisfy
\begin{equation}\label{eq3.4a}
\varepsilon <u_{\varepsilon, n},\psi>_X + \int_{\mathbb{R}^3} \nabla u_{\varepsilon,n} \cdot \nabla \psi\, dx = \lambda_{\varepsilon,n} \int_{\mathbb{R}^3} g_{\varepsilon_n} (x)\, u_{\varepsilon,n}\, \psi\, dx,\;\;\; \mbox{for any}\,\,\, \psi \in X.
\end{equation}

Assume that, $u_* \equiv 0$. Then, we consider the eigenvalue problems (\ref{eq3.3}). From Theorem
\ref{seigen} we have the existence of a curve, $\phi_{n}$, of eigenpairs for the corresponding to $g_{\varepsilon_n}$ perturbed problems:
\begin{equation}\label{eq3.4}
\varepsilon <v_{\varepsilon, n},\psi>_X + \int_{\mathbb{R}^3} \nabla v_{\varepsilon,n} \cdot \nabla \psi\, dx = \lambda_{\varepsilon,n} \int_{\mathbb{R}^3} g_{\varepsilon_n} (x)\, v_{\varepsilon,n}\, \psi\, dx,\;\;\; \mbox{for any}\,\,\, \psi \in X.
\end{equation}
However, $u_{\varepsilon_n} \rightharpoonup 0$ in $X$, so Lemma \ref{lemgc1} implies that $v_{n} \to u_0$ in $X$.. Then, Proposition \ref{mainprop} implies that there exists an interval $[0,s_*]$, depending on $u_0$, such that the set $\Phi$ is compact in $\mathbb{R} \times X$.

Choose now, $\varepsilon_n$ small enough, such that $\varepsilon_n \in [0,s_*]$ and denote by $\bar{u}_{\varepsilon_n}$ the normalization of $u_{\varepsilon_n}$ in $X$. Then, dividing (\ref{eq3.1}) by $||u_{\varepsilon_n}||_X$ we get that $\bar{u}_{\varepsilon_n}$ is a solution of (\ref{eq3.4}) with $\varepsilon_n \in [0,s_*]$. We claim that $(\lambda_{\varepsilon_n}, \bar{u}_{\varepsilon_n}) \in \phi_n$, for each $n$ big enough. Assume the opposite; then problem (\ref{eq3.4}) for the same $\varepsilon_n$ and $g_{\varepsilon_n}$, admits two nonnegative eigenfunctions corresponding to different eigenvalues. These eigenvalues should be orthogonal in $L^{2}_{g_{\varepsilon_n}} (\mathbb{R}^3)$, which is impossible, since $g_{\varepsilon_n}$ is positive. (See Remark \ref{altproof}). Thus, $(\lambda_{\varepsilon_n}, \bar{u}_{\varepsilon_n})$ must belong to $\phi_n$, for each $n$ big enough. This means that $\lambda_* = \lambda_0$ and $\bar{u}_{\varepsilon_n} \to u_0$ (strongly) in $X$. However, this is true if and only if $||u_{\varepsilon_n}||_X \to 0$. Thus, $u_{\varepsilon_n}$ converges strongly to $0$ in $X$. This proves the second alternative of the Lemma. \vspace{0.2cm}

Assume now that $u_* \not\equiv 0$. Taking the limit as $n \to \infty$, (\ref{eq3.4a}) implies that
\begin{equation} \label{eq3.5}
\int_{\mathbb{R}^3} \nabla u_* \cdot \nabla \psi\, dx = \lambda_{*}\, \int_{\mathbb{R}^3} g_*\, u_*\, \psi\, dx,
\end{equation}
where
\begin{equation} \label{g}
g_{*} (x) = h + \frac{1}{\lambda_*}\, H(\lambda_*,u_*)\, \frac{u_*}{u_* + \theta}.
\end{equation}
However, $u_* \in X$ and is nonnegative. Thus, is the unique positive eigenfunction of the above eigenvalue problem. Suppose now that $\liminf ||u_{\varepsilon_n}||_X > ||u_*||_X$, then there exists a subsequence of $u_{\varepsilon_n}$, again denote it by $u_{\varepsilon_n}$, such that
$||u_{\varepsilon_n}||_X \to M > ||u_*||_X$. Set $\tilde{u}_{\varepsilon_n} = u_{\varepsilon_n}/||u_{\varepsilon_n}||_X$. Is clear that $\tilde{u}_{\varepsilon_n}$ converges
weakly to some $\tilde{u}_{*}$ in $X$. Then, diving (\ref{eq3.1}) by $||u_{\varepsilon_n}||_X$ and taking the limit, we obtain that $\tilde{u}_{*}$ is also a nonnegative
eigenfunction of (\ref{eq3.5}), which is impossible. Thus, $\liminf ||u_{\varepsilon_n}||_X = ||u_*||_X$ and $u_{\varepsilon_n} \to u_*$, strongly, in $X$. This proves the
first alternative of the Lemma. Thus, the proof is completed.\ $\blacksquare$ \vspace{0.2cm}
%
%
\begin{remark} \label{altproof}
In the proof of the second alternative, we used that $g_{\varepsilon_n}$ is positive. However, this is not something crucial. Next we give a different proof for the second
alternative; Assume that $u_* \equiv 0$. We will prove that $\lambda_{*} = \lambda_0$, thus $(\lambda_{\varepsilon_n}, \bar{u}_{\varepsilon_n}) \in \phi_n$, (using the
notation of the above proof).

From Lemma \ref{lemgc1} we have that $\lambda_* \ne 0$ and from (\ref{eq2.19}) we have that $\tilde{u}_* \not\equiv 0$ and the limit of $A_{\varepsilon}$ is finite. Since
\[
||\varepsilon \tilde{u}_{\varepsilon_n}||^2_{X} = \varepsilon_n\, A_{\varepsilon_n} \to 0,
\]
we have that $\varepsilon_n \tilde{u}_{\varepsilon_n}$ converges in $X$ and this limit must be zero. Then, for any $\phi \in X$, we have that
\begin{equation}\label{eq2.19a}
\varepsilon <\tilde{u}_{\varepsilon_n}, \phi>_X + \int_{\mathbb{R}^3} \nabla \tilde{u}_{\varepsilon_n} \cdot \nabla \phi\, dx = \lambda_{\varepsilon_n} \int_{\mathbb{R}^3}
g_{\varepsilon_n}\, \tilde{u}_{\varepsilon_n}\, \phi\, dx,
\end{equation}
which in the limit gives
\[
\int_{\mathbb{R}^3} \nabla \tilde{u}_* \cdot \nabla \phi\, dx = \lambda_* \int_{\mathbb{R}^3} \tilde{u}_*\, \phi\, dx.
\]
However, this makes sense if and only if $(\tilde{u}_*,\lambda_*)=(\tilde{u}_{0}, \lambda_{0})$, where $\tilde{u}_{0}$ is the normalized in $D^{1,2}(\mathbb{R}^3)$ eigenfunction of
(\ref{eq1.5}) corresponding to the principal eigenvalue $\lambda_{0}$. Thus, $(\lambda_{\varepsilon_n}, \bar{u}_{\varepsilon_n}) \in \phi_n$ and the rest of the proof follows
exactly as in the proof of the above Lemma.
\end{remark}
In the next result, we actually prove that for $\varepsilon \to 0$ the nonpositive solutions $(\lambda,u)$ belonging to the branches
$\mathcal{C}_{\varepsilon}$, must tend to infinity in $\mathbb{R} \times X$.
\begin{lemma} \label{vanish}
Let $(\lambda_n, u_n)$ be a sequence belonging to
$C_{\varepsilon_n}$, such that $u_n(x) \geq 0$ and there exist
$\{a^i_n\} \subset \mathbb{R}^3$, with $u_n(a^i_n)=0$. Then, $(\lambda_n, u_n)$ is unbounded in $\mathbb{R} \times
X$.
\end{lemma}
\emph{Proof} Assume the opposite i.e., there exists $M>0$, such
that $||(\lambda_n, u_n)||_{\mathbb{R} \times X} <M$,
for any $n \in \mathbb{N}$. Then, $\lambda_n \to \lambda_*$ and $u_n \rightharpoonup u_*$, up to some subsequence. From Lemma \ref{basic argument}, we have that if $u_*
\equiv 0$, the normalization of $u_n$ must converge to $u_0$, in $X$ and thus in $L^{\infty}$, which contradicts the positivity of $u_0$. If on the other, $u_* \not\equiv 0$,
we get that $u_n \to u_*$, strongly in $X$ and $u_*$ is a solution of (\ref{eq1.5}), with $g_*$ given by (\ref{g}). The contradiction now follows from maximum principle.
$\blacksquare$
\begin{lemma} \label{notubounded} Let $\varepsilon_n \to 0$. Then, every branch $\mathcal{C}_{\varepsilon_n}$ cannot be uniformly
bounded, i.e., we cannot find $M>0$ such that $||(\lambda,
u)||_{\mathbb{R} \times X}<M$, for any $(\lambda,u)
\in \mathcal{C}_{\varepsilon_n}$.
\end{lemma}
\emph{Proof} Assume the opposite. Then, for some $\varepsilon_n \to 0$
there exists $M>0$, such that $||(\lambda,u)||_{\mathbb{R} \times
X}<M$, for any $(\lambda,u) \in
\mathcal{C}_{\varepsilon_n}$. Then, we have that the branches are
compact (i.e., the second alternative of Theorem \ref{eglobal}
holds), which means that must contain solutions of (\ref{eq1.3})
that change sign. The connectness of these branches in
$X$ and thus in $L^{\infty}(\mathbb{R}^3)$ implies that
there exist $(\lambda_{n}, u_n) \in \mathcal{C}_{\varepsilon_n}$ with $u_n$
vanishing somewhere in $\mathbb{R}^3$ and being positive
elsewhere, with $||u_n||_{X} < M$. However, this is
impossible from Lemma \ref{vanish}.\ $\blacksquare$ \vspace{0.2cm}

Next result, which is immediate, will be used in the proof of Theorem \ref{global}.
\begin{lemma} \label{eclosed}
Fix $\varepsilon$ small enough. Then the branch $\mathcal{C}_{\varepsilon}$ is a closed set in $\mathbb{R} \times
X $.
\end{lemma}

We now prove the main result of this Section; the
existence of a global branch of solutions for (\ref{eq1.1})
bifurcating from the principal eigenvalue $\lambda_0$ of
(\ref{eq1.2}). \vspace{0.3cm} \\
\emph{Proof of Theorem \ref{global}}\ \ We apply Whyburn's Lemma
\ref{whyburn} in order to prove that $\mathcal{C}_{\varepsilon}$
converge in $\mathbb{R} \times X$, and this limit $C_0$ is the
global branch of solutions of (\ref{eq1.1})
bifurcating from the principal eigenvalue $\lambda_0$ of
(\ref{eq1.2}). For some $R>0$ and some sequence $\varepsilon_n \to
0$, as $n \to \infty$, we define the sets $A_n$ as follows:
\[
A_n = \biggr \{ \bar{B}_R (\lambda_0 , 0) \cap
\mathcal{C}_{\varepsilon_n} \biggr \} \subset \mathbb{R} \times
X,
\]
For every $n \in \mathbb{N}$, these sets are connected (see
Theorem \ref{eglobal}) and closed (see Lemma \ref{eclosed}). Next,
we claim that $\liminf_{n \to \infty} \{A_n\}$ is not empty. To see this we consider the points
$(\lambda_{0,\varepsilon_n},0)$ belonging to $\mathcal{C}_{\varepsilon_n}$. From Theorem \ref{seigen}
we have that $(\lambda_{0,\varepsilon_n},0) \to (\lambda_0,0)$, hence,
\begin{equation}\label{inf}
\liminf_{n \to \infty} \{A_n\} \not\equiv \emptyset.
\end{equation}
Alternately, we may assume a sequence belonging to the set
\[
\biggr \{ \partial B_R (\lambda_0 , 0) \cap \mathcal{C}_{\varepsilon_n} \biggr \} \subset A_n.
\]
Each sequence belonging in this set is bounded and Lemma \ref{vanish}, Lemma
\ref{basic argument} imply that this sequence converges strongly in $\mathbb{R} \times X$, to some $(\lambda_*,u_*)$ which satisfy
(\ref{eq1.1}). Hence (\ref{inf}) is true.

It remains to prove that the set $\bigcup_{n \in \mathbb{N}} A_n$
is relatively compact i.e., every sequence in $A_n$ contains a
convergent subsequence. Let $(\lambda_n,u_n) \in \bigcup_{n \in
\mathbb{N}} A_n$, then the sequence $(\lambda_n,u_n)$ is bounded
in $\mathbb{R} \times X$ and so (up to a subsequence) we have that $\lambda_n \to
\lambda_*$ and $u_n \rightharpoonup u_*$ in $X$.
However, since $u_n$ is bounded, Lemma \ref{vanish} implies that
$u_n$ are positive in $\mathbb{R}^3$ and Lemma
\ref{basic argument} implies that $(\lambda_n,u_n)$ converges strongly in $\mathbb{R} \times X$, either to $(\lambda_0,0)$ or to $(\lambda_*,u_*)$ which satisfy
(\ref{eq1.1}). Thus, $\bigcup_{n \in \mathbb{N}} A_n$ is relatively compact.

Then, we leave $R \to \infty$ in
order to obtain that $\mathcal{C}_{\varepsilon_n}  \to C_0$, in
$\mathbb{R} \times X$, for any $R \in \mathbb{R}$.
In order to prove that $C_0$ is unbounded in $\mathbb{R} \times
X$, we may use, the sequences
$\{(\lambda_n,u_n) \in \mathcal{C}_{\varepsilon_n} \cap
\partial B_R (\lambda_0,0)\}$ which converge to some
$(\lambda_*,u_*)$ in $\mathbb{R} \times X$, satisfies (\ref{eq1.1}), for
any $R>0$. $\blacksquare$ \vspace{0.2cm}

%

%
%
\section{Bifurcation for The Mean Curvature Equations} \label{secmc}
\emph{Proof of Theorem \ref{thmgeneralcase}}\ \ We apply Whyburn's Lemma \ref{whyburn} in order to prove that $C_{\theta_n} \to C_0$, where $C_0$ will denote the global
branch of solutions of (\ref{mc}) bifurcating from $\lambda_0$. For some $R>0$ and some sequence $\theta_n \to 0$, as $n \to \infty$, we define the sets $A_n$, as
follows:
\[
A_n = \biggr \{ \bar{B}_R (\lambda_0 , 0) \cap
\mathcal{C}_{\theta_n} \biggr \} \subset \mathbb{R} \times
X,
\]
For every $n \in \mathbb{N}$, these sets are connected and closed. Moreover, we note that $\liminf_{n \to
\infty} \{A_n\}$ is not empty since $(\lambda_0,0) \in C_{\varepsilon_n}$, for every $n$.

It remains to prove that the set $\cup_{n \in \mathbb{N}} A_n$ is relatively compact i.e., every sequence in $A_n$ contains a convergent subsequence. Let $(\lambda_n,u_n) \in
\cup_{n \in \mathbb{N}} A_n$. The sequence $(\lambda_n,u_n)$ is bounded in $\mathbb{R} \times X$ and so (up to a subsequence), we have that $\lambda_n \to \lambda_*$ and $u_n
\rightharpoonup u_*$ in $X$. Is sufficient to prove that if $u_* \equiv 0$ then $u_n \to u_* \equiv 0$ in $X$ and $\lambda_n \to \lambda_0$. \vspace{0.2cm} \\
Let $u_* \equiv 0$, then setting $\tilde{u}_n$ to be the normalization of $u_n$ in $D^{1,2} (\mathbb{R}^3)$, from (\ref{eq1.3}) we obtain that
\begin{equation}\label{eqg.5}
1=\int_{\mathbb{R}^3} |\nabla \tilde{u}_n|^2\, dx = \lambda_n\, \int_{\mathbb{R}^3} h\, |\tilde{u}_n|^2\, dx + \int_{\mathbb{R}^3} \frac{H\, u_n}{u_n^2+\theta_n}\, \tilde{u}_n^2\, dx.
\end{equation}
Observe that $\tilde{u}_n \rightharpoonup \tilde{u}_*$ in $D^{1,2} (\mathbb{R}^3)$, such that
\[
\int_{\mathbb{R}^3} h\, |\tilde{u}_n|^2\, dx \to \int_{\mathbb{R}^3} h\, |\tilde{u}_*|^2\, dx.
\]
Moreover, as in (\ref{H0}) we get that
\[
\int_{\mathbb{R}^3} \frac{H\, u_n}{u_n^2+\theta_n}\, \tilde{u}_n^2\, dx \to 0,
\]
as $n \to \infty$. Thus $\lambda_* \ne 0$ and  (\ref{eqg.5}) implies that $\tilde{u}_* \ne 0$, such that
\[
\int_{\mathbb{R}^3} |\nabla \tilde{u}_*|^2\, dx = \lambda_*\, \int_{\mathbb{R}^3} |\tilde{u}_*|^2\, dx.
\]
Since $\tilde{u}_*$ is nonnegative, $(\lambda_*, \tilde{u}_*)$ should coincide with the principal eigenpair $(\lambda_0, \tilde{u}_0)$ ($\tilde{u}_0$ denotes the normalization of $u_0$ in $D^{1,2}(\mathbb{R}^3)$. \vspace{0.2cm}

Observe now that (\ref{eq1.3}) may be seen as
\begin{equation}\label{eqq.6}
-\Delta u_n = \lambda_n\, g_n\, u_n,
\end{equation}
with
\[
g_n = h + \frac{1}{\lambda_n} f\, \frac{H\, u_n}{u_n^2+\theta_n}.
\]
Denote by $\bar{u}_n$ the normalization of $u_n$ in $X$. Then, $\bar{u}_n$ satisfies (\ref{eqq.6}) and using the same arguments as in Lemma \ref{lemgc1} we obtain that $\bar{u}_n \to u_0$ in $X$. However, this is true only if $u_n \to 0$, in $X$. \vspace{0.2cm}

It remains now to leave $R \to \infty$ in order to obtain that $C_{\varepsilon_n} \to C_0$, in $\mathbb{R} \times X$, for any $R$. In order to prove that $C_0$ is unbounded
in $\mathbb{R} \times X$, we may use the sequences $\{(\lambda_n, u_n) \in C_{\varepsilon_n} \cap \partial B_R (\lambda_1,0)\}$, which converge to some $(\lambda_*,u_*)$, in
$\mathbb{R} \times X$, for any $R>0$. $\blacksquare$ \vspace{0.2cm}

Theorem \ref{thmgeneralcase} applies also to the case of equation (\ref{mcgeneral}); The functions $f = f(x,p,q_i)$, $g = g(x,p,q_i)$, $f,g: \mathbb{R}^{3+1+3} \to \mathbb{R}$ are smooth enough, at least $C^{6}$. Moreover, $f(x,0,0) = c$, $c$ is a positive real constant and
\[
\int g(x,u,\nabla u)\, u\, dx = o (||u||_{D^{1,2} (\mathbb{R}^3)}),\;\;\;\;\;\; \mbox{as}\;\;\; ||u||_{D^{1,2} (\mathbb{R}^3)} \to 0,
\]
for any $u \in X$. For example if $f = \sqrt{\delta + u^2 + c^2_1 |\nabla u|^2}$, $c_1$ real constant and $g = |\nabla u|^2/f$ we obtain $\delta$- approximations of the relativistic heat equation. Other application is $\varepsilon$- approximations for equations arising in transportation problems and in radiation hydrodynamics, see \cite{bren03}.

\medskip

{\sc N. B. Zographopoulos} \vspace{0.1cm}\\
University of Military Education, Hellenic Army Academy, Department of Mathematics \& Engineering Sciences, Vari - 16673, Athens Greece, \vspace{0.1cm} \\
e-mail: nzograp@gmail.com

\

{\bf Keywords:} \ {global bifurcation, mean curvature equations, non uniformly elliptic equations} \vspace{0.1cm}

\medskip

{\bf AMS Subject Classification (2010):} \  {53A10, 35A09, 35B40, 35B32, 35J62, 35J93}

\end{document}